\documentclass[a4paper,10pt]{article}
\usepackage[T1]{fontenc}
\usepackage{amsmath,amsfonts,amsthm}
\usepackage{a4wide}
\usepackage{authblk}
\usepackage{booktabs}
\usepackage{cite}
\usepackage{comment}
\usepackage{enumerate}
\usepackage{general}
\usepackage{hyperref}
\usepackage[latin1]{inputenc}
\usepackage{mathtools}
\usepackage{nicefrac}
\usepackage{numprint} 
\usepackage{paralist}
\usepackage{pgfplots,pgfplotstable}
\usepackage{pmat}
\usepackage[font={footnotesize}]{caption}
\usepackage[font={footnotesize}]{subcaption}
\usepackage{tikz}
\usepackage{xcolor}

\usepackage[bold]{hhtensor}

\graphicspath{{./figures/pdf/}}

\usepackage{parskip}

\makeatletter
\def\thm@space@setup{%
  \thm@preskip=\parskip \thm@postskip=0pt
}
\makeatother

\newcommand{\email}[1]{\href{mailto:#1}{#1}}

\pgfqkeys{/pgfplots}{
  cycle list name = black white,
}

\newcommand{\lproj}[2][h]{\pi_{#1}^{#2}}

\newcommand{\UT}[1][k]{\underline{\vec{U}}_T^{#1}}
\newcommand{\Uh}[1][k]{\underline{\vec{U}}_h^{#1}}
\newcommand{\UhD}[1][k]{\underline{\vec{U}}_{h,0}^{#1}}

\newcommand{\Ph}[1][k]{P_h^{#1}}

\newcommand{\DT}[1][k]{D_T^{#1}}
\newcommand{\rT}[1][k+1]{\vec{r}_T^{#1}}
\newcommand{\rh}[1][k+1]{\vec{r}_h^{#1}}

\newcommand{\IT}[1][k]{\underline{\vec{I}}_T^{#1}}
\newcommand{\Ih}[1][k]{\underline{\vec{I}}_h^{#1}}

\newcommand{\ms}[1][]{\matr{\sigma}_{#1}}

\newcommand{\vu}[1][]{\vec{u}_{#1}}
\newcommand{\vv}[1][]{\vec{v}_{#1}}
\newcommand{\vw}[1][]{\vec{w}_{#1}}

\newcommand{\uvu}[1][h]{\underline{\vec{u}}_{#1}}
\newcommand{\uvv}[1][h]{\underline{\vec{v}}_{#1}}

\newcommand{\uvw}[1][h]{\underline{\vec{w}}_{#1}}

\newcommand{\tuvu}[1][h]{\widehat{\underline{\vec{u}}}_{#1}}
\newcommand{\tph}[1][h]{\widehat{p}_h}

\newcommand{\uve}[1][]{\underline{\vec{e}}_{#1}}

\newcommand{\vf}[1][]{\vec{f}_{#1}}

\newcommand{\diff}[1][]{\kappa_{#1}}
\newcommand{\ldiff}[1][]{\underline{\kappa}_{#1}}
\newcommand{\udiff}[1][]{\overline{\kappa}_{#1}}
\newcommand{\ar}{\rho_{\diff}}

\newcommand{\tF}{t_{\rm F}}
\newcommand{\dt}{\ud_t}
\newcommand{\ddt}{\delta_t}
\newcommand{\bdf}{\delta_t^{(2)}}

\newcommand{\jump}[2][F]{[#2]_{#1}}
\newcommand{\wavg}[2][F]{\{#2\}_{#1}}

\newcommand{\LT}[1][k]{\vec{L}_{T}^{#1}}
\newcommand{\vphi}{\vec{\varphi}}
\newcommand{\vpsi}{\vec{\psi}}
\newcommand{\UpT}{\vec{U}_{\partial T}}
\newcommand{\ST}[1][k]{\matr{S}_{T}^{#1}}
\newcommand{\Flux}[1][TF]{\vec{\Phi}_{#1}^{k}}
\newcommand{\flux}[1][TF]{\phi_{#1}^{k}}
\newcommand{\Rh}[1][k]{R_{\diff,h}^{#1}}

\newcommand{\charac}[1]{\chi_{#1}}

\newcommand{\dollar}{\mathsf{S}}

\newcommand{\fh}[1][\partial T]{\mathfrak{h}_{#1}}

\newcommand{\cB}[2][k]{\mathcal{B}^{#1}_{#2}}
\newcommand{\bcB}[2][k]{\vec{\mathcal{B}}^{#1}_{#2}}

\newcommand{\cPh}[1][k]{\mathcal{P}^{#1}_h}
\newcommand{\bcU}[2][k]{\vec{\mathcal{U}}^{#1}_{#2}}
\newcommand{\ubcUh}[1][k]{\underline{\vec{\mathcal{U}}}_{h,0}^{#1}}

\newcommand{\matalg}[1]{\mathsf{#1}}

\title{A nonconforming high-order method for the Biot problem on general meshes}
\author[1]{Daniele Boffi\footnote{\email{daniele.boffi@unipv.it}}}
\author[1,2]{Michele Botti\footnote{\email{michele.botti01@universitadipavia.it}}}
\author[2]{Daniele A. Di Pietro\footnote{\email{daniele.di-pietro@umontpellier.fr}}}
\affil[1]{
  Università degli Studi di Pavia, Dipartimento di Matematica ``Felice Casorati'', 27100 Pavia, Italy
}
\affil[2]{
  University of Montpellier, Institut Montpéllierain Alexander Grothendieck, 34095 Montpellier, France
}

\begin{document}
\maketitle

\begin{abstract}

In this work, we introduce a novel algorithm for the Biot problem based on a Hybrid High-Order discretization of the mechanics and a Symmetric Weighted Interior Penalty discretization of the flow.
The method has several assets, including, in particular, the support of general polyhedral meshes and arbitrary space approximation order.
Our analysis delivers stability and error estimates that hold also when the specific storage coefficient vanishes, and shows that the constants have only a mild dependence on the heterogeneity of the permeability coefficient.
Numerical tests demonstrating the performance of the method are provided.
\end{abstract}

%% \tableofcontents

\section{Introduction}

We consider in this work the quasi-static Biot's consolidation problem describing Darcian flow in a deformable saturated porous medium.
Our original motivation comes from applications in geosciences, where the support of general polyhedral meshes is crucial, e.g., to handle nonconforming interfaces arising from local mesh adaptation or Voronoi elements in the near wellbore region when modelling petroleum extraction.
Let $\Omega\subset\Real^d$, $1\le d\le 3$, denote a bounded connected polyhedral domain with boundary $\partial\Omega$ and outward normal $\normal$.
For a given finite time $\tF>0$, volumetric load $\vf$, fluid source $g$, the Biot problem consists in finding a vector-valued displacement field $\vu$ and a scalar-valued pore pressure field $p$ solution of
\begin{subequations}
  \label{eq:biot.strong}
  \begin{alignat}{2}
    \label{eq:biot.strong:mech}
    -\DIV\ms(\vu) + \alpha\GRAD p &= \vf &\qquad&\text{in $\Omega\times (0,\tF)$},
    \\
    \label{eq:biot.strong:flow}
    c_0\dt p + \DIV(\alpha\dt\vu) - \DIV(\diff\GRAD p) &= g &\qquad&\text{in $\Omega\times (0,\tF)$},
  \end{alignat}
where $c_0\ge 0$ and $\alpha>0$ are real numbers corresponding to the
constrained specific storage and Biot--Willis coefficients, respectively,
$\diff$ is a real-valued permeability field such that
$\ldiff\le\diff\le\udiff$ a.e. in $\Omega$ for given real numbers $0<\ldiff\le\udiff$, and the Cauchy stress tensor is given by
$$
\ms(\vu)\eqbydef 2\mu\GRADs\vu + \lambda\Id\DIV\vu,
$$
with real numbers $\lambda\ge 0$ and $\mu>0$ corresponding to Lamé's parameters, 
$\GRADs$ denoting the symmetric part of the gradient operator applied to vector-valued fields, and $\Id$ denoting the identity matrix of $\Real^{d\times d}$.
Equations~\eqref{eq:biot.strong:mech} and~\eqref{eq:biot.strong:flow} express, respectively, the mechanical equilibrium and the fluid mass balance.
We consider, for the sake of simplicity, the following homogeneous boundary conditions:
  \begin{alignat}{2}
    \label{eq:biot.strong:bc.u}
    \vec{u} &= \vec{0} &\qquad&\text{on $\partial\Omega\times(0,\tF)$},
    \\
    \label{eq:biot.strong:bc.p}
    \diff\GRAD p\SCAL\normal &= 0 &\qquad&\text{on $\partial\Omega\times(0,\tF)$}.
\end{alignat}
Initial conditions are set prescribing $\vu(\cdot,0)=\vu^0$ and, if $c_0>0$, $p(\cdot,0)=p^0$.
In the incompressible case $c_0=0$, we also need the following compatibility condition on $g$:
\begin{equation}\label{eq:compatibility}
  \int_\Omega g(\cdot,t) = 0\qquad\forall t\in(0,\tF),
\end{equation}
as well as the following zero-average constraint on $p$:
\begin{equation}
  \label{eq:biot.strong:zero.p}
  \int_\Omega p(\cdot,t) = 0 \qquad\forall t\in(0,\tF).
\end{equation}
\end{subequations}
For the derivation of the Biot model we refer to the seminal work of Terzaghi~\cite{Terzaghi:43} and Biot~\cite{Biot:41,Biot:55}. A
theoretical study of problem~\eqref{eq:biot.strong} can be found in~\cite{Showalter:00}.
For the precise regularity assumptions on the data and on the solution under which our a priori bounds and convergence estimates are derived, we refer to Lemma~\ref{lem:a-priori} and Theorem~\ref{thm:err.est}, respectively.

A few simplifications are made to keep the exposition as simple as possible while still retaining all the principal difficulties.
For the Biot--Willis coefficient we take
$$
\alpha=1,
$$
an assumption often made in practice.
For the scalar-valued permeability $\diff$, we assume that it is piecewise constant on a partition $P_\Omega$ of $\Omega$ into bounded open polyhedra.
The treatment of more general permeability coefficients can be done following the ideas of~\cite{Di-Pietro.Ern.ea:08}.
Also, more general boundary conditions than~\eqref{eq:biot.strong:bc.u}--\eqref{eq:biot.strong:bc.p} can be considered up to minor modifications.

Our focus is here on a novel space discretization for the Biot problem (standard choices are made for the time discretization).
Several difficulties have to be accounted for in the design of the space discretization of problem~\eqref{eq:biot.strong}:
in the context of nonconforming methods, the linear elasticity operator has to be carefully engineered to ensure stability expressed by a discrete counterpart of the Korn's inequality;
the Darcy operator has to accomodate rough variations of the permeability coefficient;
the choice of discrete spaces for the displacement and the pressure must satisfy an inf-sup condition to contribute reducing spurious pressure oscillations for small time steps combined with small permeabilities when $c_0=0$.
An investigation of the role of the inf-sup condition in the context of finite element discretizations can be found, e.g., in Murad and Loula~\cite{Murad.Loula:92,Murad.Loula:94}.
A recent work of Rodrigo, Gaspar, Hu, and Zikatanov~\cite{Rodrigo.Gaspar.ea:15} has pointed out that, even for discretization methods leading to an inf-sup stable discretization of the Stokes problem in the steady case, pressure oscillations can arise owing to a lack of monotonicity. Therein, the authors suggest that stabilizing is possible by adding to the mass balance equation an artificial diffusion term with coefficient proportional to $h^2/\tau$ (with $h$ and $\tau$ denoting, respectively, the spatial and temporal meshsizes). However, computing the exact amount of stabilization required is in general feasible only in 1 space dimension.

Several space discretization methods for the Biot problem have been considered in the literature.
Finite element discretizations are discussed, e.g., in the monograph of Lewis and Schrefler~\cite{Lewis.Schrefler:98}; cf. also references therein.
A finite volume discretization for the three-dimensional Biot problem with discontinuous physical coefficients is considered by Naumovich~\cite{Naumovich:06}.
In~\cite{Phillips.Wheeler:07,Phillips.Wheeler:07*1}, Phillips and Wheeler propose and analyze an algorithm that models displacements with continuous elements and the flow with a mixed method.
In~\cite{Phillips.Wheeler:08}, the same authors also propose a different method where displacements are instead approximated using discontinuous Galerkin methods.
In~\cite{Wheeler.Xue.ea:14}, Wheeler, Xue and Yotov study the coupling of multipoint flux discretization for the flow with a discontinuous Galerkin discretization of the displacements.
While certainly effective on matching simplicial meshes, discontinuous Galerkin discretizations of the displacements usually do not allow to prove inf-sup stability on general polyhedral meshes.

In this work, we propose a novel space discretization of problem~\eqref{eq:biot.strong} where the linear elasticity operator is discretized using the Hybrid High-Order (HHO) method of~\cite{Di-Pietro.Ern:15} (c.f. also~\cite{Di-Pietro.Ern.ea:14,Di-Pietro.Drouniou:15,Di-Pietro.Ern:16}), while the flow relies on the Symmetric Weighted Interior Penalty (SWIP) discontinuous Galerkin method of~\cite{Di-Pietro.Ern.ea:08}, see also~\cite[Chapter~4]{Di-Pietro.Ern:12}.
The proposed method has several assets:%
\begin{inparaenum}[(i)]
\item It delivers an inf-sup stable discretization on general meshes including, e.g., polyhedral elements and nonmatching interfaces;
\item it allows to increase the space approximation order to accelerate convergence in the presence of (locally) regular solutions;
\item it is locally conservative on the primal mesh, a desirable property for practitioners and key for a posteriori estimates based on equilibrated fluxes;
\item it is robust with respect to the spatial variations of the permeability coefficient, with constants in the error estimates that depend on the square root of the heterogeneity ratio;
\item it is (relatively) inexpensive: at the lowest order, after static condensation of element unknowns for the displacement, we have 4 (resp. 9) unknowns per face for the displacements + 3 (resp. 4) unknowns per element for the pore pressure in 2d (resp. 3d).
\end{inparaenum}
Finally, the proposed construction is valid for arbitrary space dimension, a feature which can be exploited in practice to conceive dimension-independent implementations.

The material is organized as follows.
In Section~\ref{sec:discretization}, we introduce the discrete setting and formulate the method.
In Section~\ref{sec:stability}, we derive a priori bounds on the exact solution for regular-in-time volumetric load and mass source.
The convergence analysis of the method is carried out in Section~\ref{sec:err.anal}.
Implementation details are discussed in Section~\ref{sec:implementation}, while numerical tests proposed in Section~\ref{sec:num.tests}.
Finally, in Appendix~\ref{sec:flux.form}, we investigate the local conservation properties of the method by identifying computable conservative normal tractions and mass fluxes.

%------------------------------------------------------------------------------%

\section{Discretization}\label{sec:discretization}

In this section we introduce the assumptions on the mesh, define the discrete counterparts of the elasticity and Darcy operators and of the hydro-mechanical coupling terms, and formulate the discretization method.

\subsection{Mesh and notation}\label{sec:setting:mesh}

Denote by ${\cal H}\subset \Real_*^+ $ a countable set of meshsizes having $0$ as its unique accumulation point.
Following~\cite[Chapter~1]{Di-Pietro.Ern:12}, we consider $h$-refined spatial mesh sequences $(\Th)_{h \in {\cal H}}$ where, for all $ h \in {\cal H} $, $\Th$ is a finite collection of nonempty disjoint open polyhedral elements $T$
such that $\closure{\Omega}=\bigcup_{T\in\Th}\closure{T}$ and $h=\max_{T\in\Th} h_T$
with $h_T$ standing for the diameter of the element $T$.
We assume that mesh regularity holds in the following sense:
For all $h\in{\cal H}$, $\Th$ admits a matching simplicial submesh $\fTh$ and there exists a real number $\varrho>0$ independent of $h$ such that, for all $h\in{\cal H}$,%
\begin{inparaenum}[(i)]
\item for all simplex $S\in\fTh$ of diameter $h_S$ and inradius $r_S$, $\varrho h_S\le r_S$ and
\item for all $T\in\Th$, and all $S\in\fTh$ such that $S\subset T$, $\varrho h_T \le h_S$.
\end{inparaenum}
A mesh sequence with this property is called regular.
It is worth emphasizing that the simplicial submesh $\fTh$ is just an analysis tool, and it is not used in the actual construction of the discretization method.
These assumptions are essentially analogous to those made in the context of other recent methods supporting general meshes; cf., e.g.,~\cite[Section~2.2]{Beirao-da-Veiga.Brezzi.ea:13*1} for the Virtual Element method.
For a collection of useful geometric and functional inequalities that hold on regular mesh sequences we refer to~\cite[Chapter~1]{Di-Pietro.Ern:12} and~\cite{Di-Pietro.Droniou:15}.

\begin{remark}[Face degeneration]\label{rem:face.deg}
  The above regularity assumptions on the mesh imply that the diameter of the mesh faces is uniformly comparable to that of the cell(s) they belong to, i.e., face degeneration is not allowed.
  Face degeneration has been considered, on the other hand, in~\cite{Cangiani.Georgoulis.ea:14} in the context of interior penalty discontinuous Galerkin methods.
  One could expect that this framework could be used herein while adapting accordingly the penalty strategy~\eqref{eq:ah} and~\eqref{eq:ch}.
  This point lies out of the scope of the present work and will be inspected in the future.
\end{remark}

To avoid dealing with jumps of the permeability inside elements, we additionally assume that, for all $h\in{\cal H}$, $\Th$ is compatible with the known partition $P_\Omega$ on which the diffusion coefficient $\diff$ is piecewise constant, so that jumps can only occur at interfaces.

We define a face $F$ as a hyperplanar closed connected subset of $\closure{\Omega}$ with positive $ (d{-}1) $-dimensional Hausdorff measure and such that%
\begin{inparaenum}[(i)]
\item either there exist $T_1,T_2\in\Th $ such that $F\subset\partial
T_1\cap\partial T_2$ (with $\partial T_i$ denoting the boundary of $T_i$) and $F$ is called an interface or 
\item there exists $T\in\Th$ such that $F\subset\partial T\cap\partial\Omega$ and $F$ is called a boundary face.
\end{inparaenum}%
Interfaces are collected in the set $\Fhi$, boundary faces in $\Fhb$, and we let $\Fh\eqbydef\Fhi\cup\Fhb$.
The diameter of a face $F\in\Fh$ is denoted by $h_F$.
For all $T\in\Th$, $\Fh[T]\eqbydef\{F\in\Fh\st F\subset\partial T\}$ denotes the set of faces contained in $\partial T$ and, for all $F\in\Fh[T]$, $\normal_{TF}$ is the unit normal to $F$ pointing out of $T$.
For a regular mesh sequence, the maximum number of faces in $\Fh[T]$ can be bounded by an integer $\Np$ uniformly in $h$.
For each interface $F\in\Fhi$, we fix once and for all the ordering for the elements $T_1,T_2\in\Th$ such that $F\subset\partial T_1\cap\partial T_2$ and we let $\normal_F\eqbydef\normal_{T_1,F}$. 
For a boundary face, we simply take $\normal_F=\normal$, the outward unit normal to $\Omega$.

For integers $l\ge 0$ and $s\ge 1$, we denote by $\Poly{l}(\Th)$ the space of
fully discontinuous piecewise polynomial functions of total degree $\le l$ on $\Th$
and by $H^s(\Th)$ the space of functions in $L^2(\Omega)$ that lie in $H^s(T)$
for all $T\in\Th$. The notation $H^s(P_\Omega)$ will also be used with
obvious meaning.
Under the mesh regularity assumptions detailed above, using~\cite[Lemma~1.40]{Di-Pietro.Ern:12} together with the results of~\cite{Dupont.Scott:80}, one can prove that there exists a real number $C_{\rm app}$ depending on $\varrho$ and $l$, but independent of $h$, such that, denoting by $\lproj{l}$ the $L^2$-orthogonal projector on $\Poly{l}(\Th)$, the following holds:
For all $s\in\{1,\ldots,l+1\}$ and all $v\in H^s(\Th)$, 
\begin{equation}
  \label{eq:approx.lproj}
  \seminorm[H^m(\Th)]{v - \lproj{l} v }
  \le 
  C_{\rm app} h^{s-m} 
  \seminorm[H^s(\Th)]{v}
  \qquad \forall m \in \{0,\ldots,s-1\}.
\end{equation}

For an integer $l\ge 0$, we consider the space
$$
C^{l}(V)\eqbydef C^{l}([0,\tF];V),
$$
spanned by $V$-valued functions that are $l$ times continuously differentiable in the time interval $[0,\tF]$.
The space $C^{0}(V)$ is a Banach space when equipped with the norm $\norm[C^{0}(V)]{\varphi}\eqbydef\max_{t\in[0,\tF]}\norm[V]{\varphi(t)}$, and the space $C^{l}(V)$ is a Banach space when equipped with the norm $\norm[C^{l}(V)]{\varphi}\eqbydef\max_{0\le m\le l}\norm[C^{0}(V)]{\dt^{m}\varphi}$.
For the time discretization, we consider a uniform mesh of the time interval $(0,\tF)$ of step $\tau\eqbydef\tF/N$ with $N\in\Natural^*$, and introduce the discrete times $t^n\eqbydef n\tau$ for all $0\le n\le N$.
For any $\varphi\in C^{l}(V)$, we set $\varphi^n\eqbydef\varphi(t^n)\in V$, and we introduce the backward differencing operator $\ddt$ such that, for all $1\le n\le N$,
\begin{equation}
  \label{eq:ddt}
  \ddt\varphi^n\eqbydef\frac{\varphi^n-\varphi^{n-1}}{\tau}\in V.
\end{equation}

In what follows, for $X\subset\closure{\Omega}$, we respectively denote by ${(\cdot,\cdot)}_X$ and $\norm[X]{{\cdot}}$ the standard inner product and norm in $L^2(X)$, with the convention that the subscript is omitted whenever $X=\Omega$. The same notation is used in the vector- and tensor-valued cases.
For the sake of brevity, throughout the paper we will often use the notation $a\lesssim b$ for the inequality $a\le Cb$ with generic constant $C>0$ independent of $h$, $\tau$, $c_{0}$, $\lambda$, $\mu$, and $\diff$, but possibly depending on $\varrho$ and the polynomial degree $k$.
We will name generic constants only in statements or when this helps to follow the proofs.

\subsection{Linear elasticity operator}
\label{sec:setting:elasticity}
The discretization of the linear elasticity operator is based on the Hybrid High-Order method of~\cite{Di-Pietro.Ern:15}.
Let a polynomial degree $k\ge 1$ be fixed.
The degrees of freedom (DOFs) for the displacement are collected in the space
\begin{equation}
  \label{eq:Uh}
  \Uh\eqbydef\left\{
  \bigtimes_{T\in\Th}\Poly{k}(T)^d
  \right\}\times\left\{
  \bigtimes_{F\in\Fh}\Poly[d-1]{k}(F)^d
  \right\}.
\end{equation}
For a generic collection of DOFs in $\Uh$ we use the notation $\uvv\eqbydef\big((\vv[T])_{T\in\Th},(\vv[F])_{F\in\Fh}\big)$.
We also denote by $\vv[h]$ (not underlined) the function of $\Poly{k}(\Th)^d$ such that $\restrto{\vv[h]}{T}=\vv[T]$ for all $T\in\Th$.
The restrictions of $\Uh$ and $\uvv$ to an element $T$ are denoted by $\UT$ and $\uvv[T] = \big(\vv[T], (\vv[F])_{F\in\Fh[T]}\big)$, respectively.
For further use, we define the reduction map $\Ih:H^1(\Omega)^d\to\Uh$ such that, for all $\vv\in H^1(\Omega)^d$, 
\begin{equation}
  \label{eq:Ih}
  \Ih\vv = \big( (\lproj[T]{k}\vv)_{T\in\Th}, (\lproj[F]{k}\vv)_{F\in\Fh} \big),
\end{equation}
where $\lproj[T]{k}$ and $\lproj[F]{k}$ denote the $L^2$-orthogonal projectors on $\Poly{k}(T)$ and $\Poly[d-1]{k}(F)$, respectively.
For all $T\in\Th$, the reduction map on $\UT$ obtained by a restriction of $\Ih$ is denoted by $\IT$.

For all $T\in\Th$, we obtain a high-order polynomial reconstruction $\rT:\UT\to\Poly{k+1}(T)^d$ of the displacement field by solving the following local pure traction problem:
For a given local collection of DOFs $\uvv[T]=\big(\vv[T],(\vv[F])_{F\in\Fh[T]}\big)\in\UT$, find $\rT\uvv[T]\in\Poly{k+1}(T)^d$ such that
\begin{equation}
  \label{eq:rT}
  (\GRADs\rT\uvv[T],\GRADs\vw)_T
  = (\GRADs\vv[T],\GRADs\vw)_T
  + \sum_{F\in\Fh[T]}(\vv[F]-\vv[T],\GRADs\vw\normal_{TF})_F
  \qquad\forall\vw\in\Poly{k+1}(T)^d.
\end{equation}
In order to uniquely define the solution to~\eqref{eq:rT}, we prescribe the conditions $\int_T\rT\uvv[T]=\int_T\vv[T]$ and $\int_T\GRADss\rT\uvv[T]=\sum_{F\in\Fh[T]}\int_{F}\frac12\left(\normal_{TF}\otimes\vv[F]-\vv[F]\otimes\normal_{TF}\right)$, where $\GRADss$ denotes the skew-symmetric part of the gradient operator.
We also define the global reconstruction of the displacement $\rh:\Uh\to\Poly{k+1}(\Th)^d$ such that, for all $\uvv[h]\in\Uh$,
\begin{equation}
  \label{eq:rh}
  \restrto{(\rh\uvv[h])}{T} = \rT\uvv[T]\qquad\forall T\in\Th.
\end{equation}
The following approximation property is proved in~\cite[Lemma~2]{Di-Pietro.Ern:15}:
For all $\vv\in H^1(\Omega)^d\cap H^{k+2}(P_\Omega)^d$,
\begin{equation}
  \label{eq:rh.approx}
  \norm{\GRADs(\rh\Ih\vv - \vv)}\lesssim h^{k+1}\norm[H^{k+2}(P_\Omega)^d]{\vv}.
\end{equation}

We next introduce the discrete divergence operator $\DT:\UT\to\Poly{k}(T)$ such that, for all $q\in\Poly{k}(T)$
\begin{subequations}
  \begin{align}\label{eq:DT}
    (\DT\uvv[T],q)_T
    &= (\DIV\vv[T],q)_T
    + \sum_{F\in\Fh[T]}(\vv[F]-\vv[T],q\normal_{TF})_F
    \\
    \label{eq:DT.bis}
    &=
    -(\vv[T],\GRAD q)_{T}
    + \sum_{F\in\Fh[T]}(\vv[F],q\normal_{TF})_F,%
  \end{align}
\end{subequations}
where we have used integration by parts to pass to the second line.
The divergence operator satisfies the following commuting property: 
For all $T\in\Th$ and all $\vv\in H^1(T)^d$,
\begin{equation}
  \label{eq:commuting.DT}
  \DT\IT\vv = \lproj[T]{k}(\DIV\vv).
\end{equation}
The local contribution to the discrete linear elasticity operator is expressed by the bilinear form $a_T$ on $\UT\times\UT$ such that, for all $\uvw[T],\uvv[T]\in\UT$,
\begin{equation}
  \label{eq:aT}
  a_T(\uvw[T],\uvv[T])
  \eqbydef 2\mu\left\{
  (\GRADs\rT\uvw[T],\GRADs\rT\uvv[T])_T
  + s_T(\uvw[T], \uvv[T])
  \right\}
  + \lambda (\DT\uvw[T], \DT\uvv[T])_T,
\end{equation}
where the stabilization bilinear form $s_T$ is such that
\begin{equation}\label{eq:sT}
  s_T(\uvw[T],\uvv[T])
  \eqbydef
  \sum_{F\in\Fh[T]} h_F^{-1} (
  \vec{\Delta}_{TF}^k\uvw[T], 
  \vec{\Delta}_{TF}^k\uvv[T]
  )_F,
\end{equation}
with face-based residual such that, for all $\uvw[T]\in\UT$,
$$
\vec{\Delta}_{TF}^k\uvw[T]\eqbydef(\lproj[F]{k}\rT\uvw[T]-\vw[F])-(\lproj[T]{k}\rT\uvw[T]-\vw[T]).
$$
The global bilinear form $a_h$ on $\Uh\times\Uh$ is assembled element-wise from local contributions:
\begin{equation}
  \label{eq:ah}
  a_h(\uvw[h],\uvv[h])\eqbydef\sum_{T\in\Th} a_T(\uvw[T],\uvv[T]).
\end{equation}
To account for the zero-displacement boundary condition~\eqref{eq:biot.strong:bc.u}, we consider the subspace
	\begin{equation}
  \label{eq:UhD}
  \UhD\eqbydef\left\{
	\uvv[h]=
  \big((\vv[T])_{T\in\Th},(\vv[F])_{F\in\Fh}\big)\in\Uh\st
  \vv[F]\equiv\vec{0}\quad\forall F\in\Fhb
  \right\}	.
\end{equation}
Define on $\Uh$ the discrete strain seminorm
\begin{equation}
  \label{eq:norm1h}
  \norm[\epsilon,h]{\uvv[h]}^2\eqbydef\sum_{T\in\Th}\norm[\epsilon,T]{\uvv[h]}^2,\qquad
  \norm[\epsilon,T]{\uvv[h]}^2\eqbydef\norm[T]{\GRADs\vv[T]}^2
  + \sum_{F\in\Fh[T]}h_F^{-1}\norm[F]{\vv[F]-\vv[T]}^2.
\end{equation}
It can be proved that $\norm[\epsilon,h]{{\cdot}}$ defines a norm on $\UhD$.
Moreover, using~\cite[Corollary~6]{Di-Pietro.Ern:15}, one has the following coercivity and boundedness result for $a_h$:
\begin{equation}
  \label{eq:ah.coer}
  \eta^{-1}(2\mu)\norm[\epsilon,h]{\uvv[h]}^2\le
  \norm[a,h]{\uvv[h]}^2\eqbydef a_h(\uvv[h],\uvv[h])
  \le\eta(2\mu+d\lambda)\norm[\epsilon,h]{\uvv[h]}^2,
\end{equation}
where $\eta>0$ is a real number independent of $h$, $\tau$ and the physical coefficients.
Additionally, we know from~\cite[Theorem~8]{Di-Pietro.Ern:15} that, for all $\vw\in H_0^1(\Omega)^d\cap H^{k+2}(P_\Omega)^d$ such that $\DIV\vw\in H^{k+1}(P_\Omega)$ and all $\uvv[h]\in\UhD$, the following consistency result holds:
\begin{equation}
  \label{eq:ah.consist}
  \left|a_h(\Ih\vw,\uvv[h])+(\DIV\ms(\vw),\vv[h])\right|
  \lesssim h^{k+1}\left(
  2\mu\norm[H^{k+2}(P_\Omega)^d]{\vw} + \lambda\norm[H^{k+1}(P_\Omega)]{\DIV\vw}
  \right)\norm[\epsilon,h]{\uvv[h]}.
\end{equation}
To close this section, we prove the following discrete counterpart of Korn's inequality.

\begin{proposition}[Discrete Korn's inequality]
  \label{prop:poincare}
  There is a real number $C_{\rm K}>0$ depending on $\varrho$ and on $k$ but independent of $h$ such that, for all $\uvv[h]\in\UhD$, recalling that $\vv[h]\in\Poly{k}(\Th)^d$ denotes the broken polynomial function such that $\restrto{\vv[h]}{T}=\vv[T]$ for all $T\in\Th$,
  \begin{equation}
    \label{eq:korn}
    \norm{\vv[h]}\le C_{\rm K}d_{\Omega}\norm[\epsilon,h]{\uvv[h]},
  \end{equation}
  where $d_{\Omega}$ denotes the diameter of $\Omega$.
\end{proposition}

\begin{proof}
  Using a broken Korn's inequality~\cite{Brenner:04} on $\Poly{k}(\Th)^d$ (this is possible since $k\ge 1$), one has
  \begin{equation}
    \label{eq:korn:1}
    d_{\Omega}^{-2}\norm{\vv[h]}^2\lesssim
    \norm{\GRADsh\vv[h]}^2 
    + \sum_{F\in\Fhi}\norm[F]{\jump{\vv[h]}}^2
    + \sum_{F\in\Fhb}\norm[F]{\restrto{\vv[h]}{F}}^2,
  \end{equation}
  where $\GRADsh$ denotes the broken symmetric gradient on $H^{1}(\Th)^{d}$.
  For an interface $F\in\Fh[T_{1}]\cap\Fh[T_{2}]$, we have introduced the jump $\jump{\vv[h]}\eqbydef\vv[T_{1}]-\vv[T_{2}]$.
  Thus, using the triangle inequality, we get $\norm[F]{\jump{\vv[h]}}\le\norm[F]{\vv[F]-\vv[T_1]}+\norm[F]{\vv[F]-\vv[T_2]}$.
  For a boundary face $F\in\Fhb$ such that $F\in\Fh[T]\cap\Fhb$ for some $T\in\Th$ we have, on the other hand, $\norm[F]{\restrto{\vv[h]}{F}}=\norm[F]{\vv[F]-\vv[T]}$ since $\vv[F]\equiv 0$ (cf.~\eqref{eq:UhD}).
  Using these relations in the right-hand side of~\eqref{eq:korn:1} and rearranging the sums yields the assertion.
\end{proof}

\subsection{Darcy operator}
\label{sec:setting:darcy}
The discretization of the Darcy operator is based on the Symmetric Weighted Interior Penalty method of~\cite{Di-Pietro.Ern.ea:08}, cf. also~\cite[Section~4.5]{Di-Pietro.Ern:12}.
At each time step, the discrete pore pressure is sought in the broken polynomial space
\begin{equation}
  \label{eq:Ph}
  \Ph\eqbydef
    \begin{cases}
      \Poly{k}(\Th) & \text{if $c_0>0$},
      \\
      \Poly[d,0]{k}(\Th) & \text{if $c_0=0$,}
    \end{cases}
\end{equation}
where we have introduced the zero-average subspace $\Poly[d,0]{k}(\Th)\eqbydef\left\{ q_h\in\Poly{k}(\Th)\st (q_h,1) = 0 \right\}$.
For all $F\in\Fhi$, we define the jump and (weighted) average operators such that, for all $\varphi\in H^{1}(\Th)$, denoting by $\varphi_T$ and $\diff[T]$ the restrictions of $\varphi$ and $\diff$ to $T\in\Th$, respectively,
\begin{equation}
  \label{eq:trace.op}
  \jump{\varphi}\eqbydef\varphi_{T_1}-\varphi_{T_2},\qquad
  \wavg{\varphi}\eqbydef\omega_{T_1}\varphi_{T_1} + \omega_{T_2}\varphi_{T_2},
\end{equation}
where $\omega_{T_1}=1-\omega_{T_2}\eqbydef\frac{\diff[T_2]}{(\diff[T_1]+\diff[T_2])}$.
Denoting by $\GRADh$ the broken gradient on $H^1(\Th)$ and letting, for all $F\in\Fhi$, $\lambda_{\diff,F}\eqbydef\frac{2\diff[T_1]\diff[T_2]}{(\diff[T_1]+\diff[T_2])}$, we define the bilinear form $c_h$ on $\Ph\times\Ph$ such that, for all $q_h,r_h\in\Ph$,
\begin{equation}
  \label{eq:ch}
  \begin{aligned}
    c_h(r_h,q_h)
    &\eqbydef
    (\diff\GRADh r_h,\GRADh q_h)
    -\sum_{F\in\Fhi}\big(
    (\wavg{\diff\GRADh r_h}\SCAL\normal_F,\jump{q_h})_F
    + (\jump{r_h},\wavg{\diff\GRADh q_h}\SCAL\normal_F)_F
    \big)
    \\
    &\qquad + \sum_{F\in\Fhi}\frac{\varsigma\lambda_{\diff,F}}{h_F}(\jump{r_h},\jump{q_h})_F,
  \end{aligned}
\end{equation}
where $\varsigma>0$ is a user-defined penalty parameter.
The fact that the boundary terms only appear on internal faces in~\eqref{eq:ch} reflects the Neumann boundary condition~\eqref{eq:biot.strong:bc.p}.
From this point on, we will assume that $\varsigma>C_{\rm tr}^2\Np$ with $C_{\rm tr}$ denoting the constant from the discrete trace inequality~\cite[Eq.~(1.37)]{Di-Pietro.Ern:12}, which ensures that the bilinear form $c_h$ is coercive
(in the numerical tests of Section~\ref{sec:num.tests}, we took $\varsigma = (N_{\partial}+0.1)k^2$).
Since the bilinear form $c_h$ is also symmetric, it defines a seminorm on $\Ph$, denoted hereafter by $\norm[c,h]{{\cdot}}$ (the map $\norm[c,h]{{\cdot}}$ is in fact a norm on $\Poly[d,0]{k}(\Th)$).

\begin{remark}[Alternative stabilization]\label{rem:BR2}
  To get rid of the dependence of the lower threshold of $\varsigma$ on $C_{\rm tr}$, one can resort to the BR2 stabilization; c.f.~\cite{Bassi.Rebay:97} and also~\cite[Section 5.3.2]{Di-Pietro.Ern:12}.
  In passing, this stabilization could also contribute to handle face degeneration since the penalty parameter no longer depends on the inverse of the face diameter (cf. Remark~\ref{rem:face.deg}). This topic will make the object of future investigations.
\end{remark}

The following known results will be needed in the analysis.
Let 
$$
P_*\eqbydef\left\{r\in H^1(\Omega)\cap H^2(P_\Omega)\st\text{$\diff\GRAD r\SCAL\normal=0$ on $\partial\Omega$}\right\},\qquad
P_{*h}^k\eqbydef P_* + \Ph.
$$
Extending the bilinear form $c_h$ to $P_{*h}^k\times P_{*h}^k$, the following consistency result can be proved adapting the arguments of~\cite[Chapter~4]{Di-Pietro.Ern:12} to account for the homogeneous Neumann boundary condition~\eqref{eq:biot.strong:bc.p}: 
\begin{equation}
  \label{eq:ch.consist}
  \forall r\in P_*,\qquad
  -(\DIV(\diff\GRAD r), q)
  = c_h(r, q)\qquad\forall q\in P_{*h}.
\end{equation}
Assuming, additionally, that $r\in H^{k+2}(P_{\Omega})$, as a consequence of~\cite[Lemma~5.52]{Di-Pietro.Ern:12} together with the optimal approximation properties~\eqref{eq:approx.lproj} of $\lproj{k}$ on regular mesh sequences one has,
\begin{equation}
  \label{eq:ch.approx}
  \sup_{q_h\in\Poly[d,0]{k}(\Th)\setminus\{0\}}\frac{c_h(r-\lproj{k}r, q_h)}{\norm[c,h]{q_h}}
  \lesssim \udiff^{\nicefrac12} h^k\norm[H^{k+1}(P_{\Omega})]{r}.
\end{equation}

\subsection{Hydro-mechanical coupling}
\label{sec:coupling}
The hydro-mechanical coupling is realized by means of the bilinear form $b_h$ on $\Uh\times\Poly{k}(\Th)$ such that, for all $\uvv[h]\in\Uh$ and all $q_h\in\Poly{k}(\Th)$,
\begin{equation}
  \label{eq:bh}
  b_h(\uvv[h],q_h)
  \eqbydef \sum_{T\in\Th} b_{T}(\uvv[T],\restrto{q_{h}}{T}),\qquad
  b_{T}(\uvv[T],\restrto{q_{h}}{T})
  \eqbydef- (\DT\uvv[T],\restrto{q_h}{T})_T,
\end{equation}
where $\DT$ is the discrete divergence operator defined by~\eqref{eq:DT}.
A simple verification shows that, for all $\uvv[h]\in\Uh$ and all $q_h\in\Poly{k}(\Th)$,
\begin{equation}\label{eq:bh.cont}
b_h(\uvv[h],q_h)\lesssim\norm[\epsilon,h]{\uvv[h]}\norm{q_h}.
\end{equation}
Additionally, using the definition~\eqref{eq:DT} of $\DT$ and~\eqref{eq:UhD} of $\UhD$, it can be proved that, for all $\uvv[h]\in\UhD$, it holds ($\charac{\Omega}$ denotes here the characteristic function of $\Omega$),
\begin{equation}\label{eq:bh.const}
  b_{h}(\uvv[h],\charac{\Omega})=0.
\end{equation}
The following inf-sup condition expresses the stability of the hydro-mechanical coupling:

\begin{lemma}[inf-sup condition for $b_h$]\label{lem:inf-sup}
  There is a real number $\beta$ depending on $\Omega$, $\varrho$ and $k$ but independent of $h$ such that, for all $q_h\in\Poly[d,0]{k}(\Th)$,
  \begin{equation}
    \label{eq:inf-sup}
    \norm{q_h}\le\beta
    \sup_{\uvv[h]\in\UhD\setminus\{\underline{\vec{0}}\}}\frac{b_h(\uvv[h],q_h)}{\norm[\epsilon,h]{\uvv[h]}}.
  \end{equation}
\end{lemma}

\begin{proof}
  Let $q_h\in\Poly[d,0]{k}(\Th)$.
  Classically \cite{Boffi.Brezzi.ea:13}, there is $\vv[q_h]\in H_0^1(\Omega)^d$ such that $\DIV\vv[q_h]=q_h$ and $\norm[H^1(\Omega)^d]{\vv[q_h]}\lesssim\norm{q_h}$.
  Let $T\in\Th$. Using the $H^1$-stability of the $L^2$-orthogonal projector (cf., e.g.,~\cite[Corollary~3.7]{Di-Pietro.Droniou:15}), it is inferred that
$$
  \norm[T]{\GRADs\lproj[T]{k}\vv[q_h]}
  \le\norm[T]{\GRAD\vv[q_h]}.
$$
Moreover, for all $F\in\Fh[T]$, using the boundedness of $\lproj[F]{k}$ and the continuous trace inequality of~\cite[Lemma~1.49]{Di-Pietro.Ern:12} followed by a local Poincaré's inequality for the zero-average function $(\lproj[T]{k}\vv[q_h]-\vv[q_h])$, we have 
$$
h_F^{-\nicefrac12}\norm[F]{\lproj[F]{k}(\lproj[T]{k}\vv[q_h]-\vv[q_h])}
\le h_F^{-\nicefrac12}\norm[F]{\lproj[T]{k}\vv[q_h]-\vv[q_h]}
\lesssim\norm[T]{\GRAD\vv[q_h]}.
$$
As a result, recalling the definition~\eqref{eq:Ih} of the local reduction map $\IT$ and~\eqref{eq:norm1h} of the strain norm $\norm[\epsilon,T]{{\cdot}}$, it follows that $\norm[\epsilon,T]{\IT\vv[q_h]}\lesssim\norm[H^1(T)^d]{\vv[q_h]}$. Squaring and summing over $T\in\Th$ the latter inequality, we get
\begin{equation}
\label{bnd1h_H1}
\norm[\epsilon,h]{\Ih\vv[q_h]}\lesssim\norm[H^1(\Omega)^d]{\vv[q_h]}\lesssim\norm{q_h}.
\end{equation}
  Using~\eqref{bnd1h_H1}, the commuting property~\eqref{eq:commuting.DT}, and denoting by $\dollar$ the supremum in~\eqref{eq:inf-sup}, one has
  $$
  \norm{q_h}^2
  = (\DIV\vv[q_{h}],q_{h})
  = \sum_{T\in\Th}(\DT\IT\vv[q_h],q_h)_T
  = -b_h(\Ih\vv[q_h],q_h)
  \le\dollar\norm[\epsilon,h]{\Ih\vv[q_h]}
  \lesssim\dollar\norm{q_h}.\qedhere
  $$
\end{proof}

\subsection{Formulation of the method}\label{sec:discrete}

For all $1\le n\le N$, the discrete solution $(\uvu[h]^n,p_h^n)\in\UhD\times\Ph$ at time $t^n$ is such that, for all $(\uvv[h],q_{h})\in\UhD\times\Poly{k}(\Th)$,
\begin{subequations}
  \label{eq:biot.h}
  \begin{alignat}{2}
    \label{eq:biot.h:mech}
    a_h(\uvu[h]^n,\uvv[h]) + b_h(\uvv[h],p_h^n) &= l_h^n(\uvv[h]),
    \\
    \label{eq:biot.h:flow}
    (c_0\ddt p_h^n,q_h) -b_h(\ddt\uvu[h]^n,q_h) + c_h(p_h^n,q_h) &= (g^n,q_h),
  \end{alignat}
\end{subequations}
where the linear form $l_h^n$ on $\Uh$ is defined as
\begin{equation}
  \label{eq:lh}
  l_h^n(\uvv[h])\eqbydef(\vf^n,\vv[h])=\sum_{T\in\Th}(\vf^n,\vv[T])_T.
\end{equation}
In petroleum engineering, the usual way to enforce the initial condition is to compute a displacement from an initial (usually hydrostatic) pressure distribution.
For a given scalar-valued initial pressure field $p^{0}\in L^{2}(\Omega)$, we let $\tph^0\eqbydef\lproj{k}p^0$ and set $\uvu[h]^{0}=\tuvu[h]^{0}$ with $\tuvu[h]^{0}\in\UhD$ unique solution of
\begin{equation}
  \label{eq:tuvu0}
  a_h(\tuvu[h]^0,\uvv[h])
  = l_h^0(\uvv[h]) - b_h(\uvv[h], \tph^0)\qquad
  \forall\uvv[h]\in\UhD.
\end{equation}
If $c_0=0$, the value of $\tph^0$ is only needed to enforce the initial condition on the displacement while, if $c_{0}>0$, we also set $p_{h}^{0}=\tph^{0}$ to initialize the discrete pressure.

\begin{remark}[Discrete compatibility condition for $c_0=0$]\label{rem:comp.cond}
  Also when $c_0=0$ it is possible to take the test function $q_{h}$ in~\eqref{eq:biot.h:flow} in the full space $\Poly{k}(\Th)$ instead of the zero-average subspace $\Poly[d,0]{k}(\Th)$, since the compatibility condition is verified at the discrete level.
  To check it, it suffices to let $q_{h}=\charac{\Omega}$ in~\eqref{eq:biot.h:flow}, observe that the right-hand side is equal to zero since $g^{n}$ has zero average on $\Omega$ (cf.~\eqref{eq:compatibility}), and use the definition~\eqref{eq:ch} of $c_{h}$ together with~\eqref{eq:bh.const} to prove that the left-hand side also vanishes.
  This remark is crucial to ensure the local conservation properties of the method detailed in Section~\ref{sec:flux.form}.
\end{remark}

\section{Stability analysis}\label{sec:stability}

In this section we study the stability of problem~\eqref{eq:biot.h} and prove its well-posedness.
We recall the following discrete Gronwall's inequality, which is a minor variation of~\cite[Lemma~5.1]{Heywood.Rannacher:90}.

\begin{lemma}[Discrete Gronwall's inequality]\label{lem:disc.gronwall}
  Let an integer N and reals $\delta,G>0$, and $K\ge 0$ be given, and let $(a^n)_{0\le n\le N}$, $(b^n)_{0\le n\le N}$, and $(\gamma^n)_{0\le n\le N}$ denote three sequences of nonnegative real numbers such that, for all $0\le n\le N$
  $$
  a^n + \delta\sum_{m=0}^n b^m + K \le \delta\sum_{m=0}^n\gamma^ma^m + G.
  $$
  Then, if $\gamma^m\delta<1$ for all $0\le m\le N$, letting $\varsigma^m\eqbydef (1-\gamma^m\delta)^{-1}$, it holds, for all $0\le n\le N$,
  \begin{equation}
    \label{eq:disc.gronwall}
    a^n + \delta\sum_{m=0}^n b^m + K \le \exp\left(
    \delta\sum_{m=0}^n\varsigma^m\gamma^m
    \right)\times G.
  \end{equation}
\end{lemma}

\begin{lemma}[A priori bounds] \label{lem:a-priori}
  Assume $\vf\in C^1(L^2(\Omega)^d)$ and $g\in C^0(L^2(\Omega))$, and let $(\uvu[h]^{0},p_{h}^{0})=(\tuvu[h]^{0},\tph^{0})$ with $(\tuvu[h]^{0},\tph^{0})$ defined as in Section~\ref{sec:discrete}.
  For all $1\le n\le N$, denote by $(\uvu[h]^n,p_h^n)$ the solution to~\eqref{eq:biot.h}.
  Then, for $\tau$ small enough, it holds that
  \begin{multline}\label{eq:a-priori}
  \norm[a,h]{\uvu[h]^N}^2
    + \norm{c_{0}^{\nicefrac12}p_{h}^{N}}^{2}
    + \frac{1}{2\mu+d\lambda}\norm{p_h^N-\overline{p}_h^N}^2
    + \sum_{n=1}^N \tau \norm[c,h]{p_h^n}^2
    \lesssim
    \left((2\mu)^{-1}+c_0\right)\norm{p^0}^2
    \\
    + (2\mu)^{-1}d_{\Omega}^{2}\norm[C^1(L^2(\Omega)^d)]{\vf}^2
    + (2\mu+d\lambda)\tF^{2}\norm[C^0(L^2(\Omega))]{g}^2    
    + c_0^{-1}\tF^{2}\norm[C^0(L^2(\Omega))]{\overline{g}}^2,
  \end{multline}
  with the convention that $c_0^{-1}\norm[C^0(L^2(\Omega))]{\overline{g}}^2=0$ if $c_{0}=0$ and, 
  for $0\le n\le N$, $\overline{p}_h^n \eqbydef (p_h^n,1)$.
\end{lemma}

\begin{remark}[Well-posedness]
Owing to linearity, the well-posedness of~\eqref{eq:biot.h} is an immediate consequence of Lemma~\ref{lem:a-priori}.
\end{remark}

\begin{remark}[A priori bound for $c_{0}=0$]
  When $c_{0}=0$, the choice~\eqref{eq:Ph} of the discrete space for the pressure ensures that $\overline{p}_{h}^{n}=0$ for all $0\le n\le N$. 
  Thus, the third term in the left-hand side of~\eqref{eq:a-priori} yields an estimate on $\norm{p_{h}^{N}}^{2}$, and the a priori bound reads
    \begin{multline}
    \norm[a,h]{\uvu[h]^N}^2
    + \frac{1}{2\mu+d\lambda}\norm{p_h^N}^2
    + \sum_{n=1}^N \tau \norm[c,h]{p_h^n}^2
    \lesssim 
    \\
    (2\mu)^{-1}\left(
      d_{\Omega}^{2}\norm[C^1(L^2(\Omega)^d)]{\vf}^2 + \norm{p^0}^2 
    \right)
    + (2\mu+d\lambda)\tF^{2}\norm[C^0(L^2(\Omega))]{g}^2.
  \end{multline}
  The convention $c_0^{-1}\norm[C^0(L^2(\Omega))]{\overline{g}}^2=0$ if
$c_{0}=0$ is justified since the term $\term_{2}$ in point (4) of the following proof vanishes in this case thanks to the compatibility condition~\eqref{eq:compatibility}.
\end{remark}

\begin{proof}[Proof of Lemma~\ref{lem:a-priori}]
  Throughout the proof, $C_i$ with $i\in\Natural^*$ will denote a generic positive constant independent of $h$, $\tau$, and of the physical parameters $c_{0}$, $\lambda$, $\mu$, and $\diff$.
  \\
  (1) \emph{Estimate of $\norm{p_h^n-\overline{p}_h^n}$.} Using the inf-sup
condition~\eqref{eq:inf-sup} followed by~\eqref{eq:bh.const} to infer that
$b_{h}(\uvv[h],\overline p_{h}^{n})=0$, the mechanical equilibrium
equation~\eqref{eq:biot.h:mech}, and the second inequality
in~\eqref{eq:ah.coer}, for all $1\le n\le N$ we get
  $$
  \begin{aligned}
    \norm{p_h^n-\overline{p}_h^n}
    &\le \beta
    \sup_{\uvv[h]\in\UhD\setminus\{\underline{\vec{0}}\}}\frac{b_h(\uvv[h],p_h^n-\overline{p}_h^n)}{\norm[\epsilon,h]{\uvv[h]}}
    =
      \beta\sup_{\uvv[h]\in\UhD\setminus\{\underline{\vec{0}}\}}\frac{b_h(\uvv[h],p_h^n)}{\norm[\epsilon,h]{\uvv[h]}
    }
    \\
    &= \beta\sup_{\uvv[h]\in\UhD\setminus\{\underline{\vec{0}}\}}\frac{l_h^n(\uvv[h])-a_h(\uvu[h]^n,\uvv[h])}{\norm[\epsilon,h]{\uvv[h]}}
    \le C_1^{\nicefrac12}\left(
    d_{\Omega}\norm{\vf^n}
    +	(2\mu+d\lambda)^{\nicefrac12}\norm[a,h]{\uvu[h]^n}
    \right),
  \end{aligned}
  $$
  where we have set, for the sake of brevity, $C_{1}^{\nicefrac12}\eqbydef\beta\max(C_{\rm K},\eta)$.
  This implies, in particular,
  \begin{equation}
    \label{eq:stability:bnd.phn}
      \norm{p_h^n - \overline{p}_h^n}^2
      \le 2C_1\left(
      d_{\Omega}^{2}\norm{\vf^n}^2 + (2\mu + d\lambda)\norm[a,h]{\uvu[h]^n}^2
      \right)
  \end{equation}
  \\
  (2) \emph{Energy balance.}
  Adding~\eqref{eq:biot.h:mech} with $\uvv[h]=\tau\ddt\uvu[h]^n$ to~\eqref{eq:biot.h:flow} with $q_h=\tau p_h^n$, and summing the resulting equation over $1\le n\le N$, it is inferred
  \begin{equation}
    \label{eq:stability:1}
    \sum_{n=1}^N \tau a_h(\uvu[h]^n,\ddt\uvu[h]^n)
    + \sum_{n=1}^N \tau (c_0\ddt p_h^n, p_h^n)
    + \sum_{n=1}^N \tau \norm[c,h]{p_h^n}^2
    = \sum_{n=1}^N \tau l_h^n(\ddt\uvu[h]^n)
    + \sum_{n=1}^N \tau (g^n,p_h^n).
  \end{equation}
  We denote by $\mathcal{L}$ and $\mathcal{R}$ the left- and right-hand side of~\eqref{eq:stability:1} and proceed to find suitable lower and upper bounds, respectively.
  \\
  (3) \emph{Lower bound for $\mathcal{L}$.}
    Using twice the formula 
    \begin{equation}\label{eq:magic:BE}
    2x(x-y) = x^2 + (x-y)^2 - y^2,
    \end{equation} and telescoping out the appropriate summands, the first two terms in the left-hand side of~\eqref{eq:stability:1} can be rewritten as, respectively,
  \begin{equation}
    \label{eq:stability:2}
    \begin{aligned}
      \sum_{n=1}^N \tau a_h(\uvu[h]^n,\ddt\uvu[h]^n)
      &=\frac12\norm[a,h]{\uvu[h]^N}^2
      + \frac12\sum_{n=1}^N\tau^2\norm[a,h]{\ddt\uvu[h]^n}^2
      -\frac12\norm[a,h]{\uvu[h]^0}^2,
      \\
      \sum_{n=1}^N \tau (c_0\ddt p_h^n, p_h^n)
      &=\frac12\norm{c_0^{\nicefrac12}p_h^N}^2
      +\frac12\sum_{n=1}^N\tau^2\norm{c_{0}^{\nicefrac12}\ddt p_h^n}^2
      -\frac12\norm{c_0^{\nicefrac12}p_h^0}^2.
    \end{aligned}
  \end{equation}
  Using the above relation together with~\eqref{eq:stability:bnd.phn} and $\norm{\vf^N}\le\norm[C^1(L^2(\Omega)^d)]{\vf}$, it is inferred that
  \begin{multline}
    \label{eq:stability:bnd.L}
    \frac14\norm[a,h]{\uvu[h]^N}^2
    -\frac12\norm[a,h]{\uvu[h]^0}^2
      +\frac12\norm{c_0^{\nicefrac12}p_h^N}^2
      -\frac12\norm{c_0^{\nicefrac12}p_h^0}^2
    \\
    + \frac1{8 C_1(2\mu+d\lambda)}\norm{p_h^N - \overline{p}_h^N}^2
    + \sum_{n=1}^N \tau \norm[c,h]{p_h^n}^2
    \le
    \mathcal{L}
    + \frac{d_{\Omega}^{2}}{4(2\mu+d\lambda)}\norm[C^1(L^2(\Omega)^d)]{\vf}^2.
    \end{multline}
  \\
  (4) \emph{Upper bound for $\mathcal{R}$.}
  For the first term in the right-hand side of~\eqref{eq:stability:1}, discrete integration by parts in time yields
  \begin{equation}
    \label{eq:stability:3}
    \sum_{n=1}^N \tau l_h^n(\ddt\uvu[h]^n)
    = (\vf^N,\vu[h]^N) - (\vf^0,\vu[h]^0) - \sum_{n=1}^{N}\tau(\ddt\vf^n,\vu[h]^{n-1}),
  \end{equation}
  hence, using the Cauchy--Schwarz inequality, the discrete Korn's inequality followed by~\eqref{eq:ah.coer} to estimate $\norm{\vu[h]^n}^{2}\le \frac{C_2 d_{\Omega}^{2}}{\mu}\norm[a,h]{\uvu[h]^n}^{2}$ for all $1\le n\le N$ (with $C_{2}\eqbydef C_{\rm K}^{2}\eta/2$), and Young's inequality, one has
  \begin{equation}
    \label{eq:stability:bnd.R1}
    \begin{aligned}
      \left|\sum_{n=1}^N \tau l_h^n(\ddt\uvu[h]^n)\right|
      &\le\frac18\left(
      \norm[a,h]{\uvu[h]^N}^2
      + \norm[a,h]{\uvu[h]^0}^2
      + \frac1{2\tF}\sum_{n=1}^{N}\tau\norm[a,h]{\uvu[h]^{n-1}}^2
      \right)
      \\
      &\qquad
      + \frac{C_2 d_{\Omega}^{2}}{\mu}\left(
      \norm{\vf^N}^2
      + \norm{\vf^0}^2
      + 2\tF\sum_{n=1}^N\tau\norm{\ddt\vf^n}^2        
      \right)
      \\
      &\le
      \frac18\left(
      \norm[a,h]{\uvu[h]^N}^2
      + \norm[a,h]{\uvu[h]^0}^2
      + \frac1{2\tF}\sum_{n=0}^{N}\tau\norm[a,h]{\uvu[h]^n}^2
      \right)
      + \frac{C_2C_3d_{\Omega}^{2}}{\mu}\norm[C^1(L^2(\Omega)^d)]{\vf}^2,
    \end{aligned}
  \end{equation}
  where we have used the classical bound $\norm{\vf^N}^2  + \norm{\vf^0}^2 + 2\tF\sum_{n=1}^N\tau\norm{\ddt\vf^n}^2\le C_3\norm[C^1(L^2(\Omega)^d)]{\vf}^2$ to conclude.
  We proceed to estimate the second term in the right-hand side of~\eqref{eq:stability:1} by splitting it into two contributions as follows (here, $\overline{g}^{n}\eqbydef(g^{n},1)$):
\begin{equation}
  \label{eq:stability:split.g}
    \sum_{n=1}^N \tau (g^n,p_h^n) = 
    \sum_{n=1}^N \tau (g^n,p_h^n-\overline{p}_h^n)
    + \sum_{n=1}^N \tau (\overline{g}^n,p_h^n) \eqbydef \term_1 + \term_2.
\end{equation} 
Using the Cauchy--Schwarz inequality, the bound $\sum_{n=1}^N\tau\norm{g^n}^2\le\tF\norm[C^0(L^2(\Omega))]{g}^2$ together with~\eqref{eq:stability:bnd.phn} and Young's inequality, it is inferred that
  \begin{equation}
    \label{eq:stability:bnd.R2}
    \begin{aligned}
      \left|\term_1\right|
      &\le \left\{
        \sum_{n=1}^N\tau\norm{g^n}^2
      \right\}^{\nicefrac12}\times\left\{
      \sum_{n=1}^N\tau\norm{p_h^n-\overline{p}_h^n}^2
      \right\}^{\nicefrac12}
      \\
      &\le
        \tF\norm[C^0(L^2(\Omega))]{g}\times
        \left\{
          \frac{2C_{1}}{\tF}\sum_{n=1}^N\tau\left(d_{\Omega}^{2}\norm{\vf^n}^2
          + (2\mu+d\lambda)\norm[a,h]{\uvu[h]^n}^2\right)
        \right\}^{\nicefrac12}
      \\
      &\le 
        8C_1\tF^{2}(2\mu+d\lambda)\norm[C^0(L^2(\Omega))]{g}^2
        + \frac{d_{\Omega}^{2}}{16(2\mu+d\lambda)}\norm[C^1(L^2(\Omega)^{d})]{\vf}^2
        + \frac{1}{16\tF}\sum_{n=1}^N\tau\norm[a,h]{\uvu[h]^n}^2.
    \end{aligned}
  \end{equation}
  Owing the compatibility condition~\eqref{eq:compatibility}, $\term_2=0$ if $c_0=0$.
  If $c_{0}>0$, using the Cauchy--Schwarz and Young's inequalities, we have
  \begin{equation}  
    \label{eq:stability:bnd.R2b}
      \left|\term_2\right|
      \le \left\{
        \tF\sum_{n=1}^N\tau c_0^{-1}\norm{\overline{g}^n}^2
      \right\}^{\nicefrac12}\hspace{-1ex}\times\left\{
        \tF^{-1}\sum_{n=1}^N\tau \norm{c_0^{\nicefrac12}p_h^n}^2
      \right\}^{\nicefrac12}\hspace{-1ex}
      \le
      \frac{\tF^{2}}{2c_0}\norm[C^0(L^2(\Omega))]{\overline{g}}^2
      +\frac1{2\tF}\sum_{n=1}^N\tau\norm{c_0^{\nicefrac12}p_h^n}^2.
  \end{equation}
Using~\eqref{eq:stability:bnd.R1},~\eqref{eq:stability:bnd.R2}, and~\eqref{eq:stability:bnd.R2b}, we infer
  \begin{multline}
    \label{eq:stability:bnd.R}
    \mathcal{R}
    \le\frac18\left(
    \norm[a,h]{\uvu[h]^N}^2
    + \tF^{-1}\sum_{n=0}^{N}\tau\norm[a,h]{\uvu[h]^n}^2
    +  \norm[a,h]{\uvu[h]^0}^2
    \right)
    +\frac1{2\tF}\sum_{n=1}^N\tau\norm{c_0^{\nicefrac12}p_h^n}^2
    + \frac{\tF^{2}}{2c_0}\norm[C^0(L^2(\Omega))]{\overline{g}}^2
    \\
    + 8C_1\tF^{2}(2\mu+d\lambda)\norm[C^0(L^2(\Omega))]{g}^2
    + \left(
      \frac{1}{16(2\mu+d\lambda)}+\frac{C_2C_3}{\mu}
    \right)d_{\Omega}^{2}\norm[C^1(L^2(\Omega)^d)]{\vf}^2.
  \end{multline}
  \\
  (5) \emph{Conclusion.}
  Using~\eqref{eq:stability:bnd.L}, the fact that $\mathcal{L}=\mathcal{R}$ owing to~\eqref{eq:stability:1}, and~\eqref{eq:stability:bnd.R}, it is inferred that
  \begin{multline}
    \label{eq:stability:5}
    \norm[a,h]{\uvu[h]^N}^2
    + 4\norm{c_0^{\nicefrac12}p_h^N}^2
    + \frac1{(2\mu+d\lambda)}\norm{p_h^N - \overline{p}_h^N}^2
    + 8\sum_{n=1}^N \tau \norm[c,h]{p_h^n}^2
    \le
    \\
    \frac{C_4}{\tF} \sum_{n=0}^{N}\tau\norm[a,h]{\uvu[h]^n}^2
    + \frac{C_4}{\tF}\sum_{n=1}^N\tau4\norm{c_0^{\nicefrac12}p_h^n}^2\
    + G,
  \end{multline}
  where $C_4 \eqbydef \max(1,C_1)$ while, observing that $\norm{c_0^{\nicefrac12}p_h^0}\le\norm{c_0^{\nicefrac12}p^0}$ since $\lproj{k}$ is a bounded operator, and that it follows from~\eqref{eq:ic.stab} below that $\norm[a,h]{\uvu^{0}}^{2}\le C_5 (2\mu)^{-1}\left(
      d_{\Omega}^{2}\norm{\vf^{0}}^{2} + \norm{p^{0}}^{2}\right)$,
  \begin{multline*}
  C_4^{-1} G
  \eqbydef 
  \frac{5C_5}{2\mu}\left( d_{\Omega}^{2}\norm{\vf^0}^2 + \norm{p^0}^2 \right)
  + 4\norm{c_0^{\nicefrac12}p^0}^2
  + \frac{4\tF^{2}}{c_0}\norm[C^0(L^2(\Omega))]{\overline{g}}^2
  \\
  + 64 C_1\tF^{2}(2\mu+d\lambda)\norm[C^0(L^2(\Omega))]{g}^2
  + \left(
  \frac{5}{2(2\mu+d\lambda)}+\frac{8C_2C_3}{\mu}
  \right)d_{\Omega}^{2}\norm[C^1(L^2(\Omega)^d)]{\vf}^2.
  \end{multline*}
  Using Gronwall's Lemma~\ref{lem:disc.gronwall} with $a^{0}\eqbydef\norm[a,h]{\uvu[h]^{0}}^{2}$ and $a^n\eqbydef\norm[a,h]{\uvu[h]^n}^2+ 4\norm{c_0^{\nicefrac12}p_h^n}^2$ for $1\le n\le N$, $\delta\eqbydef\tau$, $b^{0}\eqbydef 0$ and $b^n\eqbydef \norm[c,h]{p_h^n}^2$ for $1\le n\le N$, $K=\frac{1}{(2\mu+d\lambda)}\norm{p_h^N-\overline{p}_h^N}^2$, and $\gamma^n=\frac{C_{4}}{\tF}$, the desired result follows.
\end{proof}

\begin{proposition}[Stability and approximation properties for $\tuvu^{0}$]
  The initial displacement~\eqref{eq:tuvu0} satisfies the following stability condition:
  \begin{equation}\label{eq:ic.stab}
    \norm[a,h]{\tuvu^{0}}\lesssim (2\mu)^{-\nicefrac12}\left(
      d_{\Omega}\norm{\vf^{0}} + \norm{p^{0}}
    \right).
  \end{equation}
  Additionally, recalling the global reduction map $\Ih$ defined by~\eqref{eq:Ih}, and assuming the additional regularity $p_0\in H^{k+1}(P_\Omega)$, $\vu^{0}\in H^{k+2}(P_\Omega)^d$, and $\DIV\vu^{0}\in H^{k+1}(P_\Omega)$, it holds
  \begin{equation}\label{eq:tuvu0.approx}
    (2\mu)^{\nicefrac12}\norm[a,h]{\tuvu[h]^{0}-\Ih\vu^{0}}
    \lesssim h^{k+1}\left(
    2\mu\norm[H^{k+2}(P_\Omega)^d]{\vu^{0}}
    + \lambda\norm[H^{k+1}(P_\Omega)]{\DIV\vu^{0}}
    + \ar^{\nicefrac12}\norm[H^{k+1}(P_\Omega)]{p^{0}}
    \right).
  \end{equation}
\end{proposition}

\begin{proof}
  (1) \emph{Proof of~\eqref{eq:ic.stab}.}
  Using the first inequality in~\eqref{eq:ah.coer} followed by the definition~\eqref{eq:tuvu0} of $\tuvu^{0}$, we have
  $$
    \begin{aligned}
      \norm[a,h]{\tuvu^{0}}
      &\lesssim\sup_{\uvv[h]\in\UhD\setminus\{\vec{0}\}}\frac{a_h(\tuvu^{0},\uvv[h])}{(2\mu)^{\nicefrac12}\norm[\epsilon,h]{\uvv[h]}}
      \\
      &=(2\mu)^{-\nicefrac12}\sup_{\uvv[h]\in\UhD\setminus\{\vec{0}\}}
      \frac{l_h^0(\uvv[h])-b_h(\uvv[h],\lproj{k}p^{0})}{\norm[\epsilon,h]{\uvv[h]}}
      \lesssim (2\mu)^{-\nicefrac12}\left(
        d_{\Omega}\norm{\vf^0}+\norm{p^{0}}
      \right),
    \end{aligned}
  $$
    where to conclude we have used the Cauchy--Schwarz and discrete Korn's~\eqref{eq:korn} inequalities for the first term in the numerator and the continuity~\eqref{eq:bh.cont} of $b_h$ together with the $L^2(\Omega)$-stability of $\lproj{k}$ for the second.
    (2) \emph{Proof of~\eqref{eq:tuvu0.approx}.}
    The proof is analogous to that of point (3) in Lemma~\ref{lem:tuvu.tph.approx} except that we use the approximation properties~\eqref{eq:approx.lproj} of $\lproj{k}$ instead of~\eqref{eq:tph.approx.L2}.
    For this reason, elliptic regularity is not needed.
 \end{proof}%

%------------------------------------------------------------------------------%

\section{Error analysis}\label{sec:err.anal}

In this section we carry out the error analysis of the method.

\subsection{Projection}
\label{sec:err.anal:projection}
We consider the error with respect to the sequence of projections $(\tuvu[h]^n,\tph^n)_{1\le n\le N}$, of the exact solution defined as follows:
For $1\le n\le N$,  $\tph^n\in\Ph$ solves
\begin{subequations}\label{eq:projection}
\begin{equation}
  \label{eq:tph}
  c_h(\tph^n,q_h) = c_h(p^n,q_h)\qquad\forall q_h\in\Poly{k}(\Th),
\end{equation}
with the closure condition $\int_\Omega \tph^n = \int_\Omega{p^n}$.
Once $\tph^n$ has been computed, $\tuvu[h]^n\in\UhD$ solves
\begin{equation}
  \label{eq:tuvu}
  a_h(\tuvu[h]^n,\uvv[h])
  = l_h^n(\uvv[h]) - b_h(\uvv[h], \tph^n)\qquad
  \forall\uvv[h]\in\UhD.
\end{equation}
\end{subequations}
The well-posedness of problems~\eqref{eq:tph} and~\eqref{eq:tuvu} follow, respectively, from the coercivity of $c_h$ on $\Poly[d,0]{k}(\Th)$ and of $a_h$ on $\UhD$.
The projection $(\tuvu[h]^n,\tph^n)$ is chosen so that a convergence rate of $(k+1)$ in space analogous to the one derived in~\cite{Di-Pietro.Ern:15} can be proved for the $\norm[a,h]{{\cdot}}$-norm of the displacement at final time $\tF$.
To this purpose, we also need in what follows the following elliptic regularity, which holds, e.g., when $\Omega$ is convex:
There is a real number $C_{\rm ell}>0$ only depending on $\Omega$ such that, for all $\psi\in L^{2}_{0}(\Omega)$, with $L^{2}_{0}(\Omega)\eqbydef \left\{q\in L^{2}(\Omega) \st (q,1)=0 \right\}$, the unique function $\zeta\in H^1(\Omega)\cap L^{2}_{0}(\Omega)$  solution of the homogeneous Neumann problem
\begin{equation}
  \label{eq:ell.reg.prob}
  -\DIV(\diff\GRAD\zeta)=\psi\quad\text{in $\Omega$},
  \qquad
  \diff\GRAD\zeta\SCAL\normal=0\quad\text{on $\partial\Omega$},
\end{equation}
is such that 
\begin{equation}\label{eq:ell.reg}
  \norm[H^2(P_\Omega)]{\zeta}\le C_{\rm ell}\ldiff^{-\nicefrac12}\norm{\psi}.
\end{equation}
For further insight on the role of the choice~\eqref{eq:projection} and of the elliptic regularity assumption we refer to Remark~\ref{rem:projection}.

\begin{lemma}[{Approximation properties for $(\tuvu[h]^n,\tph^n)$}] \label{lem:tuvu.tph.approx}
  Let a time step $1\le n\le N$ be fixed.
  Assuming the regularity $p^{n}\in H^{k+1}(P_\Omega)$, it holds
  \begin{equation}
    \label{eq:tph.approx.en}
    \norm[c,h]{\tph^n-p^n} 
    \lesssim
    h^k \udiff^{\nicefrac12}\norm[H^{k+1}(P_{\Omega})]{p^n}.
  \end{equation}
  Moreover, recalling the global reduction map $\Ih$ defined by~\eqref{eq:Ih}, further assuming the regularity $\vu^{n}\in H^{k+2}(P_\Omega)^d$, $\DIV\vu^{n}\in H^{k+1}(P_\Omega)$, and provided that the elliptic regularity~\eqref{eq:ell.reg} holds, one has
  \begin{align}
    \label{eq:tph.approx.L2}
    \norm{\tph^n-p^n}
    &\lesssim h^{k+1}\ar^{\nicefrac12}\norm[H^{k+1}(P_\Omega)]{p^n},
    \\
    \label{eq:tuvu.approx}
    (2\mu)^{\nicefrac12}\norm[a,h]{\tuvu[h]^n-\Ih\vu^n}
    &\lesssim h^{k+1}\left(
    2\mu\norm[H^{k+2}(P_\Omega)^d]{\vu^n}
    + \lambda\norm[H^{k+1}(P_\Omega)]{\DIV\vu^n}
    + \ar^{\nicefrac12}\norm[H^{k+1}(P_\Omega)]{p^n}
    \right).
  \end{align}
  with global heterogeneity ratio $\ar\eqbydef\nicefrac{\udiff}{\ldiff}$.
\end{lemma}

\begin{proof}
  (1) \emph{Proof of~\eqref{eq:tph.approx.en}.}
  By definition, we have that $\norm[c,h]{\tph^{n}-p^{n}}=\inf_{q_h\in\Poly{k}(\Th)}\norm[c,h]{q_h-p^{n}}$.
  To prove~\eqref{eq:tph.approx.en}, it suffices to take $q_h=\lproj{k}p^{n}$ in the right-hand side of the previous expression and use the approximation properties~\eqref{eq:approx.lproj} of $\lproj{k}$.
  \\
    (2) \emph{Proof of~\eqref{eq:tph.approx.L2}.} Let $\zeta\in H^1(\Omega)$ solve~\eqref{eq:ell.reg.prob} with $\psi=p^n-\tph^n$. 
  From the consistency property~\eqref{eq:ch.consist}, it follows that
  $$
  \norm{\tph^n-p^n}^2
  = -(\DIV(\diff\GRAD\zeta),\tph^n-p^n)
  = c_h(\zeta,\tph^n-p^n)
  = c_h(\zeta-\lproj{1}\zeta,\tph^n-p^n).
  $$
  Then, using the Cauchy--Schwarz inequality, the estimate~\eqref{eq:tph.approx.en} together with the approximation properties~\eqref{eq:approx.lproj} of $\lproj{1}$, and elliptic regularity, it is inferred that
  $$
  \begin{aligned}
    \norm{\tph^n-p^n}^2
    &=
    c_h(\zeta-\lproj{1}\zeta,\tph^n-p^n)
    \le\norm[c,h]{\zeta-\lproj{1}\zeta}\norm[c,h]{\tph^n-p^n}
    \\
    &\lesssim h^{k+1}\udiff^{\nicefrac12}\norm[H^2(P_\Omega)]{\zeta}\norm[H^{k+1}(P_\Omega)]{p^n}
    \lesssim h^{k+1}\ar^{\nicefrac12}\norm{\tph^n-p^n}\norm[H^{k+1}(P_\Omega)]{p^n},
  \end{aligned}
  $$
  and~\eqref{eq:tph.approx.L2} follows.
  \\
  (3) \emph{Proof of~\eqref{eq:tuvu.approx}.}
  We start by observing that
  \begin{equation}
    \label{eq:tuvu.approx:1}
    \norm[a,h]{\tuvu[h]^n-\Ih\vu^n}
    =
    \sup_{\uvv[h]\in\Uh\setminus\{\underline{\vec{0}}\}}
    \frac{a_h(\tuvu[h]^n-\Ih\vu^n,\uvv[h])}{\norm[a,h]{\uvv[h]}}
    \lesssim
    \sup_{\uvv[h]\in\Uh\setminus\{\underline{\vec{0}}\}}
    \frac{a_h(\tuvu[h]^n-\Ih\vu^n,\uvv[h])}{(2\mu)^{\nicefrac12}\norm[\epsilon,h]{\uvv[h]}},
  \end{equation}
  where we have used the first inequality in~\eqref{eq:ah.coer}. 
  Recalling the definition~\eqref{eq:lh} of the linear form $l_h^n$, the fact that $\vf^n=-\DIV\ms(\vu) + \GRAD p$, and using~\eqref{eq:tph}, it is inferred that
  \begin{equation}
    \label{eq:tuvu.approx:base}
    \begin{aligned}
      a_h(\tuvu[h]^n-\Ih\vu^n,\uvv[h])
      &= l_h^n(\uvv[h]^n) - a_h(\Ih\vu^n,\uvv[h]) - b_h(\uvv[h],\tph{n})
      \\
      &= 
      \big\{
      -a_h(\Ih\vu^n,\uvv[h]) - (\DIV\ms(\vu^n),\vv[h])
      \big\}
      + \big\{
      (\GRAD p^n,\vv[h]) - b_h(\uvv[h],\tph^{n})
      \big\}.
    \end{aligned}
  \end{equation}
  Denote by $\term_1$ and $\term_2$ the terms in braces.
  Using~\eqref{eq:ah.consist}, it is readily inferred that 
  \begin{equation}
    \label{eq:tuvu.approx:T1}
    |\term_1|\lesssim h^{k+1}\left(
    2\mu\norm[H^{k+2}(P_\Omega)^d]{\vu^n}
    + \lambda\norm[H^{k+1}(P_\Omega)]{\DIV\vu^n}
    \right)\norm[\epsilon,h]{\uvv[h]}.
  \end{equation}
  For the second term, performing an element-wise integration by parts on $(\GRAD p,\vv[h])$ and recalling the definition~\eqref{eq:bh} of $b_h$ and~\eqref{eq:DT} of $\DT$ with $q=\tph^{n}$, it is inferred that
  \begin{equation}
    \label{eq:tuvu.approx:T2}
    \begin{aligned}
      |\term_2| 
      &= \left|\sum_{T\in\Th}\left\{
      (\tph^{n}-p^n,\DIV\vv[T])_T
      + \sum_{F\in\Fh[T]}(\tph^{n}-p^n,(\vv[F]-\vv[T])\normal_{TF})_F
      \right\}
      \right|
      \\
      &\lesssim h^{k+1}\ar^{\nicefrac12}\norm[H^{k+1}(P_\Omega)]{p^n}\norm[\epsilon,h]{\uvv[h]},
    \end{aligned}
  \end{equation}
  where the conclusion follows from the Cauchy--Schwarz inequality together with~\eqref{eq:tph.approx.L2}.
  Plugging~\eqref{eq:tuvu.approx:T1}--\eqref{eq:tuvu.approx:T2} into~\eqref{eq:tuvu.approx:base} we obtain~\eqref{eq:tuvu.approx}.
\end{proof}

\subsection{Error equations}\label{sec:err.anal:err.eq}

We define the discrete error components as follows: For all $1\le n\le N$,
\begin{equation}
  \label{eq:err.comp}
  \uve[h]^n\eqbydef\uvu[h]^n-\tuvu[h]^n,\qquad
  \rho_h^n\eqbydef p_h^n-\tph^n.
\end{equation}
Owing to the choice of the initial condition detailed in Section~\ref{sec:discrete}, the inital error $(\uve[h]^0,\rho_h^0) \eqbydef(\uvu[h]^{0}-\tuvu[h]^{0},p_h^0-\tph^0)$ is the null element in the product space $\UhD\times\Ph$.
On the other hand, for all $1\le n\le N$, $(\uve[h]^n,\rho_h^n)$ solves
\begin{subequations}
  \label{eq:err.eq}
  \begin{alignat}{2}
    \label{eq:err.eq:mech}
    a_h(\uve[h]^n,\uvv[h]) + b_h(\uvv[h],\rho_h^n) &= 0 
    &\qquad&\forall\uvv[h]\in\Uh,
    \\
    \label{eq:err.eq:flow}
    (c_0\ddt\rho_h^n,q_h)-b_h(\ddt\uve[h]^n,q_h) + c_h(\rho_h^n,q_h) &= \mathcal{E}_h^n(q_h),
    &\qquad&\forall q_h\in\Ph,
  \end{alignat}
\end{subequations}
with consistency error 
\begin{equation}
  \label{eq:Eh}
  \mathcal{E}_h^n(q_h)
  \eqbydef
  (g^n,q_h) -(c_0\ddt\tph^n,q_h) - c_h(\tph^n,q_h) + b_h(\ddt\tuvu[h]^n,q_h).
\end{equation}
\subsection{Convergence}
\label{sec:err.anal:conv}

\begin{theorem}[Estimate for the discrete errors] \label{thm:err.est}
  Let $(\vu,p)$ denote the unique solution to~\eqref{eq:biot.strong}, for which we assume the regularity
  \begin{equation}
    \label{eq:reg.up}
    \vu\in C^2(H^1(P_\Omega)^d)\cap C^1(H^{k+2}(P_\Omega)^d),\qquad
    p\in C^1(H^{k+1}(P_\Omega)).
  \end{equation}
  If $c_0>0$, we further assume $p\in C^2(L^2(\Omega))$. 
  Define, for the sake of brevity, the bounded quantities
  $$
  \begin{aligned}
    \mathcal{N}_1&\eqbydef
    \left( 2\mu+d\lambda \right)^{\nicefrac12} 
     \norm[C^2(H^1(P_\Omega)^d)]{\vu} 
     + \norm[C^2(L^2(\Omega)^d)]{c_0^{\nicefrac12} p},
    \\
    \mathcal{N}_2&\eqbydef
    \frac{(2\mu+d\lambda)^{\nicefrac12}}{2\mu}
    \left(
    2\mu\norm[C^1(H^{k+2}(P_\Omega)^d)]{\vu}
    +\lambda\norm[C^1(H^{k+1}(P_\Omega))]{\DIV\vu}
    + \ar^{\nicefrac12}\norm[C^1(H^{k+1}(P_\Omega))]{p}
  \right)
  \\
  &\qquad + \norm[C^0(H^{k+1}(P_{\Omega}))]{c_0^{\nicefrac12}p}.
  \end{aligned}
  $$
  Then, assuming the elliptic regularity~\eqref{eq:ell.reg}, it holds, letting $\overline{\rho}_{h}^{n}\eqbydef(\rho_{h}^{n},1)$,
  \begin{equation}
    \label{eq:err.est}
    \norm[a,h]{\uve[h]^N}^2
    + \norm{c_{0}^{\nicefrac12}\rho_{h}^{N}}^{2}
    + \frac{1}{2\mu+d\lambda}\norm{\rho_h^N-\overline{\rho}_h^N}^2
    +\sum_{n=1}^N\tau\norm[c,h]{\rho_h^n}^2
    \lesssim
    \left(\tau\mathcal{N}_1+h^{k+1}\mathcal{N}_2\right)^2.
  \end{equation}
\end{theorem}

\begin{remark}[Pressure estimate for $c_0=0$]
  In the incompressible case $c_0=0$, the third term in the left-hand side of~\eqref{eq:err.est} delivers an estimate on the $L^2$-norm of the pressure 
since $\overline{\rho}_h^N=0$ (cf.~\eqref{eq:biot.strong:zero.p}).
\end{remark}

\begin{proof}[Proof of Theorem~\ref{thm:err.est}]
  Throughout the proof, $C_i$ with $i\in\Natural^*$ will denote a generic positive constant independent of $h$, $\tau$, and of the physical parameters $c_{0}$, $\lambda$, $\mu$, and $\diff$.
  \\
(1) \emph{Basic error estimate.}
Using the inf-sup condition~\eqref{eq:inf-sup}, equation~\eqref{eq:bh.const} followed by~\eqref{eq:err.eq:mech}, and the second inequality in~\eqref{eq:ah.coer}, it is readily seen that 
\begin{equation}
\label{eq:err.average.inf-sup}
    \norm{\rho_h^n-\overline{\rho}_h^n}
    \le \beta
    \sup_{\uvv[h]\in\UhD\setminus\{\underline{\vec{0}}\}}\frac{b_h(\uvv[h],\rho_h^n-\overline{\rho}_h^n)}{\norm[\epsilon,h]{\uvv[h]}}
    =  \beta\sup_{\uvv[h]\in\UhD\setminus\{\underline{\vec{0}}\}}\frac{-a_h(\uve[h]^n,\uvv[h])}{\norm[\epsilon,h]{\uvv[h]}}
    \le C_1^{\nicefrac12}(2\mu+d\lambda)^{\nicefrac12}\norm[a,h]{\uve[h]^n},
\end{equation}
with $C_1^{\nicefrac12}=\beta\eta^{\nicefrac12}$.
Adding~\eqref{eq:err.eq:mech} with $\uvv[h]=\tau\ddt\uve[h]$ to~\eqref{eq:err.eq:flow} with $q_{h}=\tau\rho_{h}^{n}$ and summing the resulting equation over $1\le n \le N$, it is inferred that
\begin{equation}
    \label{eq:err.eq.sum}
    \sum_{n=1}^N\tau a_h{(\uve[h]^n,\ddt\uve[h]^n)}
    + \sum_{n=1}^N\tau(c_0\ddt\rho_h^n,\rho_h^n)
    + \sum_{n=1}^N\tau\norm[c,h]{\rho_h^n}^2
    =
    \sum_{n=1}^N\tau\mathcal{E}_h^n(\rho_h^n).
  \end{equation}
Proceeding as in point (3) of the proof of Lemma~\ref{lem:a-priori}, and recalling that $(\uve[h]^{0},\rho_h^0)=(\underline{\vec{0}},0)$, we arrive at the following error estimate:
  \begin{equation}
    \label{eq:err.est:1}
    \frac14\norm[a,h]{\uve[h]^N}^2
    + \frac{1}{4C_1(2\mu+d\lambda)}\norm{\rho_h^N-\overline{\rho}_h^N}^2
    + \frac12\norm{c_0^{\nicefrac12}\rho_h^N}^2
    + \sum_{n=1}^N\tau\norm[c,h]{\rho_h^n}^2
    \le
    \sum_{n=1}^N\tau\mathcal{E}_h^n(\rho_h^n).
  \end{equation}
  \\
  (2) \emph{Bound of the consistency error.}
  Using $g^n=c_0\dt p^n+\DIV(\dt\vu^n-\diff\GRAD p^n)$, the consistency property~\eqref{eq:ch.consist}, and observing that, using the definition~\eqref{eq:ch} of $c_{h}$, integration by parts together with the homogeneous displacement boundary condition~\eqref{eq:biot.strong:bc.u}, and~\eqref{eq:bh.const},
  $$
  c_h(p^n-\tph^n,\overline\rho_h^n)
  + (\DIV(\dt\vu^{n}),\overline\rho_{h}^{n})
  + b_h(\ddt\tuvu[h]^n,\overline\rho_h^n)=0,
  $$
  we can decompose the right-hand side of~\eqref{eq:err.est:1} as follows:
  \begin{equation}
    \label{eq:err.est:Ehn.dec}
    \begin{aligned}
        \sum_{n=1}^N\tau\mathcal{E}_h^n(\rho_h^n)
        &=  
        \sum_{n=1}^N\tau (c_0(\dt p^n-\ddt\tph^n),\rho_h^n)
        + \sum_{n=1}^N \tau c_h(p^n-\tph^n,\rho_h^n-\overline\rho_h^n)
        \\
        &\qquad 
        + \sum_{n=1}^N\tau\left\{
        (\DIV(\dt\vu^n),\rho_h^n-\overline\rho_h^n) + b_h(\ddt\tuvu[h]^n,\rho_h^n-\overline\rho_h^n)
        \right\}
        \eqbydef\term_1 + \term_2 + \term_3.
    \end{aligned}
  \end{equation}
  For the first term, inserting $\pm\ddt p^n$ into the first factor and using the Cauchy-Schwarz inequality followed by the approximation properties of $\tph^{0}$ (a consequence  of~\eqref{eq:approx.lproj}) and~\eqref{eq:tph.approx.L2} of $\tph^n$, it is inferred that
  \begin{equation}
  \label{eq:err.est:est.T1}
      \begin{aligned}
        \left|\term_1\right| 
        &\lesssim 
        \left\{ 
          c_{0}\sum_{n=1}^N\tau \left[\norm{\dt p^n-\ddt p^n}^2 + \norm{\ddt (p^n-\tph^n)}^2\right] 
        \right\}^{\nicefrac12} \times\left\{
          \sum_{n=1}^N\tau\norm{c_{0}^{\nicefrac12}\rho_h^n}^2
        \right\}^{\nicefrac12}
        \\
        &\le 
        C_2\left(
          \tau\mathcal{N}_1 + h^{k+1}\mathcal{N}_2
          \right)
        +\frac12\sum_{n=1}^N\tau\norm{c_{0}^{\nicefrac12}\rho_h^n}^2.
      \end{aligned}
  \end{equation}
For the second term, the choice~\eqref{eq:tph} of the pressure projection readily yields 
  \begin{equation}\label{eq:err.est:T2}
    \term_{2}=0.
  \end{equation}
For the last term, inserting $\pm\Ih \vu^n$ into the first argument of $b_h$, and using the commuting property~\eqref{eq:commuting.DT} of $\DT$, it is inferred that
  $$
  \term_3
  =\sum_{n=1}^N\tau\left\{
  \sum_{T\in\Th}\left[
  (\DIV(\dt\vu^n-\ddt\vu^n),\rho_h^n-\overline\rho_h^n)_T + (\DT\ddt(\IT\vu^n-\tuvu[T]^n),\rho_h^n-\overline\rho_h^n)_T
  \right]
  \right\}.
  $$
  Using the Cauchy--Schwarz inequality, the bound $\norm[T]{\DT\ddt(\IT\vu^n-\tuvu[T]^n)}\lesssim\norm[\epsilon,T]{\ddt(\IT\vu^n-\tuvu[T]^n)}$ valid for all $T\in\Th$, and the approximation properties~\eqref{eq:tuvu0.approx} and~\eqref{eq:tuvu.approx} of $\tuvu^{0}$ and $\tuvu^{n}$, respectively, we obtain
  \begin{equation}
    \begin{aligned}
    \label{eq:err.est:est.T3}
      \left|\term_3\right|
      &\lesssim\left\{
      \sum_{n=1}^N\tau\left[
      \norm[H^1(\Omega)^d]{\dt\vu^n-\ddt\vu^n}^2
      + \norm[\epsilon,h]{\ddt(\Ih\vu^n-\tuvu[h]^n)}^2
      \right]
      \right\}^{\nicefrac12}\times\left\{
      \sum_{n=1}^N\tau\norm{\rho_h^n-\overline\rho_h^n}^2
      \right\}^{\nicefrac12}
      \\
      &\le 
      C_3C_1\left(\tau\mathcal{N}_1 + h^{k+1}\mathcal{N}_2\right)^2
      +\frac{1}{4C_1(2\mu+d\lambda)}\sum_{n=1}^N\tau\norm{\rho_h^n-\overline\rho_h^n}^2.
    \end{aligned}
  \end{equation}
  Using~\eqref{eq:err.est:est.T1}--\eqref{eq:err.est:est.T3} to bound the right-hand side of~\eqref{eq:err.est:Ehn.dec}, it is inferred  
  \begin{multline}
    \norm[a,h]{\uve[h]^N}^2
    +\frac{1}{C_1(2\mu+d\lambda)}\norm{\rho_h^N-\overline\rho_h^N}^2
    + 2\norm{c_0^{\nicefrac12}\rho_h^N}^2
    +4\sum_{n=1}^N\tau\norm[c,h]{\rho_h^n}^2
    \\
    \le    
    \frac{1}{C_1(2\mu+d\lambda)}\sum_{n=1}^N\tau\norm{\rho_h^n-\overline\rho_h^n}^2
    +2\sum_{n=1}^N\tau\norm{c_0^{\nicefrac12}\rho_h^n}^2
    +G,
  \end{multline}
  with $G\eqbydef 4(C_1 C_3+C_2)\left(\tau\mathcal{N}_1+ h^{k+1}\mathcal{N}_2\right)^2$.
  The conclusion follows using the discrete Gronwall's inequality~\eqref{eq:disc.gronwall} with $\delta=\tau$, $K=\norm[a,h]{\uve[h]^N}^2$, $a^0=0$ and $a^n=\frac{1}{C_1(2\mu+d\lambda)}\norm{\rho_h^n-\overline\rho_h^n}^2+2\norm{c_0^{\nicefrac12}\rho_h^n}^2$ for $1\le n\le N$, $b^n=4\norm[c,h]{\rho_h^n}^2$, and $\gamma^n=1$.
\end{proof}

\begin{remark}[Role of the choice~\eqref{eq:projection} and of elliptic regularity]\label{rem:projection}
  The choice~\eqref{eq:projection} for the projection ensures that the term $\term_{2}$ in step (2) of the proof of Theorem~\ref{thm:err.est} vanishes.
  This is a key point to obtain an order of convergence of $(k+1)$ in space.
  For a different choice, say $\tph^{n}=\lproj{k}p^{n}$, this term would be of order $k$, and therefore yield a suboptimal estimate for the terms in the left-hand side of~\eqref{eq:en.u.L2.p} below (the estimate~\eqref{eq:en.p} would not change and remain optimal).
  This would also be the case if we removed the elliptic regularity assumption~\eqref{eq:ell.reg}.
\end{remark}

\begin{remark}[BDF2 time discretization]\label{rem:bdf2}
  In some of the numerical test cases of Section~\ref{sec:num.tests}, we have used a BDF2 time discretization, which corresponds to the backward differencing operator
\begin{equation}
\label{eq:BDF2}
  \bdf\varphi^{n+2}\eqbydef\frac{3\varphi^{n+2}-4\varphi^{n+1}+\varphi^n}{2\tau},
\end{equation}
  used in place of~\eqref{eq:ddt}. As BDF2 requires two starting values, we perform a first march in time using the backward Euler scheme
  (another possibility would have been to resort to the second-order Crank--Nicolson scheme).
  For the BDF2 time discretization, stability estimates similar to those of Lemma~\ref{lem:a-priori} can be proved with this initialization, while the error can be shown to scale as $\tau^2+h^{k+1}$ (compare with~\eqref{eq:err.est}).
  The main difference with respect to the present analysis focused on the backward Euler scheme is that formula~\eqref{eq:magic:BE} is replaced in the proofs by
  $$
  2x (3x-4y+z) = x^{2} - y^{2} + (2x-y)^{2} - (2y-z)^{2} + (x-2y+z)^{2}.
  $$
  The modifications of the proofs are quite classical and are not detailed here for the sake of conciseness (for a pedagogic exposition, one can consult, e.g.,~\cite[Chapter~6]{Ern.Guermond:04}).
\end{remark}

\begin{corollary}[Convergence]
  Under the assumptions of Theorem~\ref{thm:err.est}, it holds that
  \begin{multline}
    \label{eq:en.u.L2.p}
    (2\mu)^{\nicefrac12}\norm{\GRADsh(\rh\uvu[h]^N-\vu^N)}
    + \norm{c_0^{\nicefrac12}(p_h^N-p^N)}
    +\frac{1}{2\mu+d\lambda}
    \norm{(p_h^N-p^N)-(\overline{p}_h^N-\overline{p}^N)}
    \\
    \lesssim
    \tau \mathcal{N}_1 + h^{k+1}\mathcal{N}_2
    + c_0^{\nicefrac12} h^{k+1}\norm[H^{k+1}(P_\Omega)]{p^N},
  \end{multline}
  \begin{equation}
    \label{eq:en.p}
    \left\{
    \sum_{n=1}^N\tau\norm[c,h]{p_h^n-p^n}^2
    \right\}^{\nicefrac12}
    \lesssim 
    \tau \mathcal{N}_1 + h^{k+1}\mathcal{N}_2
    + h^k\udiff^{\nicefrac12}\tF^{\nicefrac12}\norm[C^0(H^{k+1}(P_\Omega))]{p}.
  \end{equation}
\end{corollary}

\begin{proof}
  Using the triangular inequality, recalling the definition~\eqref{eq:err.comp} of $\uve[h]^N$ and $\tph^N$ and~\eqref{eq:ah.coer} of $\norm[a,h]{{\cdot}}$-norm, it is inferred that 
  $$
  \begin{aligned}
    (2\mu)^{\nicefrac12}\norm{\GRADsh(\rh\uvu[h]^N-\vu^N)}
    &\lesssim\norm[a,h]{\uve[h]^N}
    +(2\mu)^{\nicefrac12}\norm{\GRADsh(\rh\tuvu[h]-\rh\Ih\vu^N)}
    \\
    &\qquad +(2\mu)^{\nicefrac12}\norm{\GRADs(\rh\Ih\vu^N-\vu^N)},
    \\
    \norm{p_h^N-p^N-(\overline{p}_h^N-\overline{p}^N)}
    &\le\norm{\rho_h^N-\overline{\rho}_h^N} + \norm{\tph^N-p^N},
    \\
    \norm{c_0^{\nicefrac12}(p_h^N-p^N)}
    &\le\norm{c_0^{\nicefrac12}\rho_h^N}+\norm{c_0^{\nicefrac12}(\tph^N-p^N)}.
  \end{aligned}
  $$
  To conclude, use~\eqref{eq:err.est} to estimate the left-most terms in the right-hand sides of the above equations. 
  Use~\eqref{eq:tuvu.approx} and~\eqref{eq:tph.approx.L2}, the approximation properties~\eqref{eq:rh.approx} of $\rh\Ih$, respectively, for the right-most terms.
  This proves~\eqref{eq:en.u.L2.p}. A similar decomposition of the error yields~\eqref{eq:en.p}.
\end{proof}

%------------------------------------------------------------------------------%

\section{Implementation}\label{sec:implementation}

In this section we discuss practical aspects including, in particular, static condensation.
The implementation is based on the \textsf{hho} platform\footnote{DL15105 Université de Montpellier}, which relies on the linear algebra facilities provided by the \textsf{Eigen3} library~\cite{Guennebaud.Jacob:10}.

The starting point consists in selecting a basis for each of the polynomial spaces appearing in the construction.
Let $\vec{s}=(s_1,...,s_d)$ be a $d$-dimensional multi-index with the usual notation $|\vec{s}|_{1}=\sum_{i=1}^d s_i$, and let $\vec{x}=(x_1,...,x_d)\in\Real^d$.
Given $k\ge 0$ and $T\in\Th$, we denote by $\cB{T}$ a basis for the polynomial space $\Poly{k}(T)$.
In the numerical experiments of Section~\ref{sec:num.tests}, we have used the set of locally scaled monomials:
\begin{equation}
  \label{eq:basis_fun} 
   \cB{T}\eqbydef\left\{\left(\frac{\vec{x}-\vec{x}_T}{h_T}\right)^{\vec{s}},\;|\vec{s}|_1\le k\right\},
\end{equation}
with $\vec{x}_T$ denoting the barycenter of $T$.
Similarly, for all $F\in\Fh$, we denote by $\cB{F}$ a basis for the polynomial space $\Poly[d-1]{k}(F)$ which, in the proposed implementation, is again a set of locally scaled monomials similar to~\eqref{eq:basis_fun}.

\begin{remark}[Choice of the polynomial bases]
  The choice of the polynomial bases can have a sizeable impact on the conditioning of both the local problems defining the displacement reconstruction $\rT$ (cf.~\eqref{eq:rT}) and the global problem.
  This is particularly the case when using high polynomial orders (typically, $k\ge 7$).
  The scaled monomial basis~\eqref{eq:basis_fun} is appropriate when dealing with isotropic elements.
  In the presence of anisotropic elements, a better choice is to use for each element a local frame aligned with its principal axes of rotation together with normalization factors tailored for each direction.
  A further improvement, originally investigated in~\cite{Bassi.Botti.ea:12} in the context of dG methods, consists in performing a Gram--Schmidt orthonormalization with respect to a suitably selected inner product.
  In the numerical test cases of Section~\ref{sec:num.tests}, which focus on isotropic meshes and moderate polynomial degrees ($k\le 3$), the basis~\eqref{eq:basis_fun} proved fully satisfactory.
\end{remark}%

Introducing the vector bases $\bcB{T}\eqbydef(\cB{T})^d$, $T\in\Th$, and $\bcB{F}\eqbydef(\cB{F})^d$, $F\in\Fhi$, a basis $\ubcUh$ for the space $\UhD$ (cf.~\eqref{eq:UhD}) is given by
$$
\ubcUh\eqbydef\bcU{\cal T}\times\bcU{\cal F},\qquad
\bcU{\cal T}\eqbydef\bigtimes_{T\in\Th}\bcB{T},\qquad
\bcU{\cal F}\eqbydef\bigtimes_{F\in\Fhi}\bcB{F},
$$
while a basis $\cPh$ for the space $\Ph$ (cf.~\eqref{eq:Ph}) is obtained setting
$$
\cPh\eqbydef\bigtimes_{T\in\Th}\cB{T}.
$$
When $c_0=0$, the zero average constraint in $\Ph$ can be accounted for using as a Lagrange multiplier the characteristic function of $\Omega$.
Notice also that boundary faces have been excluded from the Cartesian product in the definition of $\bcU{\cal F}$ to strongly account for boundary conditions.
Letting, for the sake of brevity, $N_n^k\eqbydef{k+n\choose k}$, $n\in\Natural$, a simple computation shows that
$$
\dim(\bcU{\cal T})=d\card{\Th}N_d^k,\qquad
\dim(\bcU{\cal F})=d\card{\Fhi}N_{d-1}^k,\qquad
\dim(\cPh)=\card{\Th}N_d^k.
$$
The total DOF count thus yields
\begin{equation}\label{eq:total.DOFs}
  d\card{\Th}N_d^k + d\card{\Fhi}N_{d-1}^k + \card{\Th}N_d^k.
\end{equation}
In what follows, for a given time step $0\le n\le N$, we denote by $\matalg{U}_{\cal T}^n$ and $\matalg{U}_{\cal F}^n$ the vectors collecting element-based and face-based displacement DOFs, respectively, and by $\matalg{P}^n$ the vector collecting pressure DOFs.

Denote now by $\matalg{A}$ and $\matalg{B}$, respectively, the matrices that represent the bilinear forms $a_h$ (cf.~\eqref{eq:ah}) and $b_h$ (cf.~\eqref{eq:bh}) in the selected basis.
Distinguishing element-based and face-based displacement DOFs, the matrices $\matalg{A}$ and $\matalg{B}$ display the following block structure:
$$
\matalg{A}=
\begin{pmat}[{|}]
  \matalg{A}_{{\cal T}{\cal T}} & \matalg{A}_{{\cal T}{\cal F}} \cr\-
  \matalg{A}_{{\cal T}{\cal F}}\trans & \matalg{A}_{{\cal F}{\cal F}} \cr
\end{pmat},\qquad
\matalg{B}=
\begin{pmat}[{.}]
  \matalg{B}_{\cal T}\vphantom{\matalg{A}_{{\cal T}{\cal T}}} \cr\-
  \matalg{B}_{\cal F}\vphantom{\matalg{A}_{{\cal T}{\cal F}}\trans} \cr
\end{pmat}.
$$
For every mesh element $T\in\Th$, the element-based displacement DOFs are only coupled with those face-based displacement DOFs that lie on the boundary of $T$ and with the (element-based) pressure DOFs in $T$.
This translates into the fact that the submatrix $\matalg{A}_{{\cal T}{\cal T}}$ is block-diagonal, i.e.,
$$
\matalg{A}_{{\cal T}{\cal T}}={\rm diag}(\matalg{A}_{TT})_{T\in\Th},
$$
with each elementary block $\matalg{A}_{TT}$ of size $\dim(\bcB{T})^2$.
Additionally, it can be proved that the blocks $\matalg{A}_{TT}$, $T\in\Th$, are invertible, so that the inverse of $\matalg{A}_{{\cal T}{\cal T}}$ can be efficiently computed setting
\begin{equation}\label{eq:ATT.inv}
  \matalg{A}_{{\cal T}{\cal T}}^{-1} = {\rm diag}(\matalg{A}_{TT}^{-1})_{T\in\Th}.
\end{equation}
The above remark can be exploited in practice to efficiently eliminate the element-based displacement DOFs from the global system.
This process, usually referred to as ``static condensation'', is detailed in what follows.

For a given time step $1\le n\le N$, the linear system corresponding to the discrete problem~\eqref{eq:biot.h} is of the form
\begin{equation}\label{eq:biot.ls}
  \begin{pmat}[{||}]
    \matalg{A}_{{\cal T}{\cal T}} & \matalg{A}_{{\cal T}{\cal F}} & \matalg{B}_{\cal T} \cr\-
    \matalg{A}_{{\cal T}{\cal F}}\trans & \matalg{A}_{{\cal F}{\cal F}} & \matalg{B}_{\cal F} \cr\-
    -\matalg{B}_{\cal T}\trans & -\matalg{B}_{\cal F}\trans & \tfrac{\tau}{\theta}\matalg{C} + c_0\matalg{M} \cr
  \end{pmat}
  \begin{pmat}[{.}]
    \matalg{U}_{\cal T}^n \cr\-
    \matalg{U}_{\cal F}^n \cr\-
    \matalg{P}^n \cr
  \end{pmat}
  =
  \begin{pmat}[{.}]
    \matalg{F}_{\cal T}^n \cr\-
    \matalg{0}_{\cal F} \cr\-
    \widetilde{\matalg{G}}^n \cr
  \end{pmat},
\end{equation}
where $\matalg{C}$ denotes the matrix that represents the bilinear form $c_h$ in the selected basis, $\matalg{M}$ is the (block diagonal) pressure mass matrix, $\matalg{F}_{\cal T}^n$ is the vector corresponding to the discretization of the volumetric load $\vf^n$, while $\matalg{0}_{\cal F}$ is the zero vector of length $\dim(\bcU{\cal F})$.
Denoting by $\matalg{G}^n$ the vector corresponding to the discretization of the fluid source $g^n$, when the backward Euler method is used to march in time, we let $\theta=1$ and set
$$
\widetilde{\matalg{G}}^n
\eqbydef
\tau\matalg{G}^n - \matalg{B}\matalg{U}^{n-1}.
$$
For the BDF2 method (and $n\ge 2$) , we let $\theta=\nicefrac32$ and set
$$
\widetilde{\matalg{G}}^n
\eqbydef
\frac23\tau\matalg{G}^n
-\frac43\matalg{B}\matalg{U}^{n-1}
-\frac13\matalg{B}{\matalg{U}^{n-2}}.
$$

Recalling~\eqref{eq:ATT.inv}, instead of assemblying the full system, we can effectively compute the Schur complement of $\matalg{A}_{{\cal T}{\cal T}}$ and code, instead, the following reduced version, where the element-based displacement DOFs collected in the subvector $\matalg{U}_{\cal T}^n$ no longer appear:
\begin{equation}\label{eq:biot.ls:reduced}
  \begin{pmat}[{|}]
    \matalg{A}_{{\cal F}{\cal F}} - \matalg{A}_{{\cal T}{\cal F}}\trans\matalg{A}_{{\cal T}{\cal T}}^{-1}\matalg{A}_{{\cal T}{\cal F}}
    & \matalg{B}_{\cal F} - \matalg{A}_{{\cal T}{\cal F}}\trans\matalg{A}_{{\cal T}{\cal T}}^{-1}\matalg{B}_{\cal T} \cr\-
    -\matalg{B}_{\cal F}\trans + \matalg{B}_{\cal T}\trans\matalg{A}_{{\cal T}{\cal T}}^{-1}\matalg{A}_{{\cal T}{\cal F}}
    & \tfrac{\tau}{\theta}\matalg{C} + c_0\matalg{M} + \matalg{B}_{\cal T}\trans\matalg{A}_{{\cal T}{\cal T}}^{-1}\matalg{B}_{\cal T} \cr
  \end{pmat}
  \begin{pmat}[{.}]
    \matalg{U}_{\cal F}^n \cr\-
    \matalg{P}^n \cr
  \end{pmat}
  =
  \begin{pmat}[{.}]
    - \matalg{A}_{{\cal T}{\cal F}}\trans\matalg{A}_{{\cal T}{\cal T}}^{-1}\matalg{F}_{\cal T}^n \cr\-
    \widetilde{\matalg{G}}^n + \matalg{B}_{\cal T}\trans\matalg{A}_{{\cal T}{\cal T}}^{-1}\matalg{F}_{\cal T}^n \cr    
  \end{pmat}.
\end{equation}
All matrix products appearing in~\eqref{eq:biot.ls:reduced} are directly assembled from their local counterparts (i.e., the factors need not be constructed separately).
Specifically, introducing, for all $T\in\Th$, the following local matrices $\matalg{A}(T)$ and $\matalg{B}(T)$ representing the local bilinear forms $a_T$ (cf.~\eqref{eq:aT}) and $b_T$ (cf.~\eqref{eq:bh}), respectively:
$$
\matalg{A}(T)=\begin{pmat}[{|}]
\matalg{A}_{TT} & \matalg{A}_{T\Fh[T]} \cr\- \matalg{A}_{T\Fh[T]}\trans & \matalg{A}_{\Fh[T]\Fh[T]} \cr
\end{pmat},\qquad
\matalg{B}(T)=\begin{pmat}[{.}]
\matalg{B}_T \cr\- \matalg{B}_{\Fh[T]} \cr
\end{pmat},
$$
one has for the left-hand side matrix, denoting by $\xleftarrow[T\in\Th]{}$ the usual assembly procedure based on a global DOF map,
$$
\begin{alignedat}{2}
  \matalg{A}_{{\cal T}{\cal F}}\trans\matalg{A}_{{\cal T}{\cal T}}^{-1}\matalg{A}_{{\cal T}{\cal F}}
  &\xleftarrow[T\in\Th]{}
  \matalg{A}_{T\Fh[T]}\trans\matalg{A}_{TT}^{-1}\matalg{A}_{T\Fh[T]},
  &\qquad
  \matalg{A}_{{\cal T}{\cal F}}\trans\matalg{A}_{{\cal T}{\cal T}}^{-1}\matalg{B}_{\cal T}
  &\xleftarrow[T\in\Th]{}
  \matalg{A}_{T\Fh[T]}\trans\matalg{A}_{TT}^{-1}\matalg{B}_T,
  \\
  \matalg{B}_{\cal T}\trans\matalg{A}_{{\cal T}{\cal T}}^{-1}\matalg{A}_{{\cal T}{\cal F}}
  &\xleftarrow[T\in\Th]{}
  \matalg{B}_{T}\trans\matalg{A}_{TT}^{-1}\matalg{A}_{T\Fh[T]},
  &\qquad
  \matalg{B}_{\cal T}\trans\matalg{A}_{{\cal T}{\cal T}}^{-1}\matalg{B}_{\cal T}
  &\xleftarrow[T\in\Th]{}
  \matalg{B}_{T}\trans\matalg{A}_{TT}^{-1}\matalg{B}_{T},
\end{alignedat}
$$
and, similarly, for the right-hand side vector
$$
\matalg{A}_{{\cal T}{\cal F}}\trans\matalg{A}_{{\cal T}{\cal T}}^{-1}\matalg{F}_{\cal T}^n
\xleftarrow[T\in\Th]{}
\matalg{A}_{T\Fh[T]}\trans\matalg{A}_{TT}^{-1}\matalg{F}_T^n,\qquad
\matalg{B}_{\cal T}\trans\matalg{A}_{{\cal T}{\cal T}}^{-1}\matalg{F}_{\cal T}^n
\xleftarrow[T\in\Th]{}
\matalg{B}_T\trans\matalg{A}_{TT}^{-1}\matalg{F}_T^n.
$$

The advantage of implementing~\eqref{eq:biot.ls:reduced} over~\eqref{eq:biot.ls} is that the number of DOFs appearing in the linear system reduces to (compare with~\eqref{eq:total.DOFs})
\begin{equation}\label{eq:coupled.DOFs}
  d\card{\Fhi}N_{d-1}^k + \card{\Th}N_d^k.
\end{equation}
Additionally, since the reduced left-hand side matrix in~\eqref{eq:biot.ls:reduced} does not depend on the time step $n$, it can be assembled (and, possibly, factored) once and for all in a preliminary stage, thus leading to a further reduction in the computational cost.
Finally, for all $T\in\Th$, the local vector $\matalg{U}_T^n$ of element-based displacement DOFs can be recovered from the local right-hand side vector $\matalg{F}_T^n$ and the local vector of face-based displacement DOFs and (element-based) pressure DOFs $(\matalg{U}_{\Fh[T]}^n,\matalg{P}^n_T)$ by the following element-by-element post-processing:
$$
\matalg{U}_T^n=\matalg{A}_{TT}^{-1}\left(
\matalg{F}_T^n - \matalg{A}_{T\Fh[T]}\matalg{U}_{\Fh[T]}^n - \matalg{B}_T\matalg{P}^n_T
\right).
$$

%------------------------------------------------------------------------------%

\section{Numerical tests}\label{sec:num.tests}

In this section we present a comprehensive set of numerical tests to assess the properties of our method.

\subsection{Convergence}
We first consider a manufactured regular exact solution to confirm the convergence rates predicted in~\eqref{eq:err.est}.
Specifically, we solve the two-dimensional incompressible Biot problem ($c_{0}=0$) in the unit square domain $\Omega=(0,1)^{2}$ with $\tF=1$ and physical parameters $\mu=1$, $\lambda=1$, and $\diff =1$.
The exact displacement $\vu$ and exact pressure $p$ are given by, respectively
  \begin{align*}
    \vu(\vec{x},t)
    &= \big(
    -\sin(\pi t)\cos(\pi x_1)\cos(\pi x_2), \sin(\pi t)\sin(\pi x_1)\sin(\pi x_2)
    \big),
    \\
    p(\vec{x},t)
    &= - \cos(\pi t)\sin(\pi x_1)\cos(\pi x_2).
  \end{align*}
  The volumetric load is given by
  $$
  \vf(\vec{x},t)
   = 6\pi^2(\sin(\pi t)+ \pi \cos(\pi t)) \times\big(
   -\cos(\pi x_1)\cos(\pi x_2), \sin(\pi x_1)\sin(\pi x_2)
   \big),
  $$
  while $g(\vec{x},t)\equiv 0$. 
  Dirichlet boundary conditions for the displacement and Neumann boundary conditions for the pressure are inferred from exact solutions to $\partial\Omega$.
  
  \begin{figure}
  \centering
    \includegraphics[height=3cm]{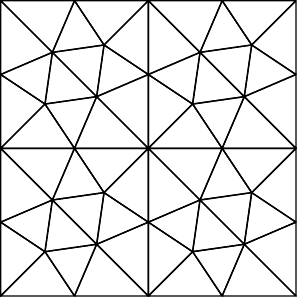}
    \hspace{0.2cm}
    \includegraphics[scale=1]{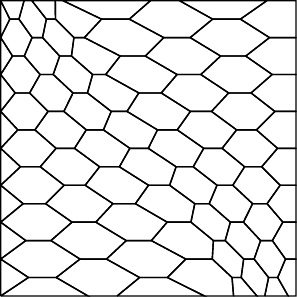}
    \hspace{0.2cm}
    \includegraphics[height=3cm]{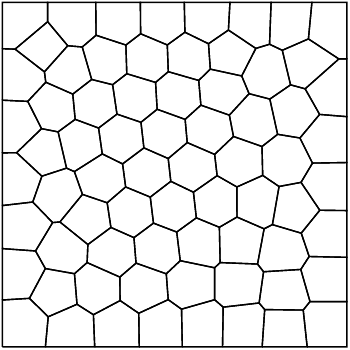}
    \hspace{0.2cm}
    \includegraphics[height=3cm]{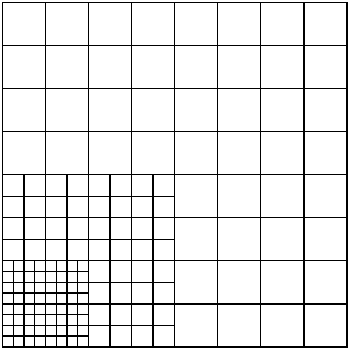}
  \caption{Triangular, hexagonal-dominant, Voronoi, and nonmatching quadrangular meshes for the numerical tests. The triangular and nonmatching quadrangular meshes were originally proposed for the FVCA5 benchmark~\cite{Herbin.Hubert:08}, whereas the hexagonal-dominant mesh is the same used in~\cite[Section~4.2.3]{Di-Pietro.Lemaire:15}. \label{fig:meshes}}
\end{figure}   
   
We consider the triangular, (predominantly) hexagonal, Voronoi, and nonmatching quadrangular mesh families depicted in 
  Figure~\ref{fig:meshes}. The Voronoi mesh family was obtained using the \textsf{PolyMesher} algorithm of~\cite{Talischi2012}.
The nonmatching mesh is simply meant to show that the method supports nonconforming interfaces: refining in the corner has no particular meaning for the selected solution.
The time discretization is based on the second order Backward Differentiation Formula (BDF2); cf.~Remark~\ref{rem:bdf2}.
The time step $\tau$ on the coarsest mesh is taken to be $\pgfmathprintnumber{0.1}/2^{\frac{(k+1)}{2}}$ for every choice of the spatial degree $k$, and it decreases with the mesh size $h$ according to the theoretical convergence rates, thus, if $h_2 = h_1/2$, then $\tau_2 = \tau_1 / 2^{\frac{(k+1)}{2}}$.
Figure~\ref{fig:convergence} displays convergence results for the various mesh families and polynomial degrees up to $3$.
The error measures are $\norm{p_{h}^{N}-\lproj{k}p^{N}}$ for the pressure and $\norm[a,h]{\uvu[h]^{N}-\Ih\vu^{N}}$ for the displacement.
Using the triangle inequality together with~\eqref{eq:err.est} and the approximation properties~\eqref{eq:approx.lproj} of $\lproj{k}$ and~\eqref{eq:rh.approx} of $(\rh\circ\Ih)$, it is a simple matter to prove that these quantities have the same convergence behaviour as the terms in the left-hand side of~\eqref{eq:err.est}.
In all the cases, the numerical results show asymptotic convergence rates that are in agreement with theoretical predictions.
This test was also used to numerically check that the mechanical equilibrium and mass conservation relations of Lemma~\ref{lem:flux_form} hold up to machine precision.
\begin{figure}
  \centering
  %% Triangular mesh 
  %% L2 pressure error
  \begin{minipage}[b]{0.45\textwidth}
    \includegraphics[scale=0.9]{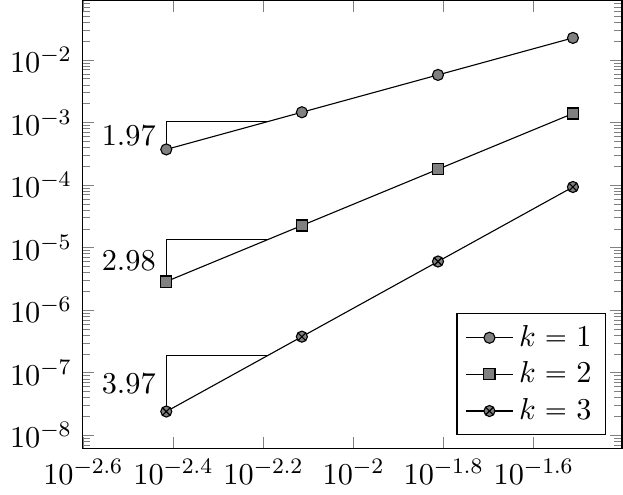}
    \subcaption{$\norm{p_h^N - \lproj{k} p^N}$, triangular}
  \end{minipage}
  %% Displacement energy error
  \begin{minipage}[b]{0.45\textwidth}
    \includegraphics[scale=0.9]{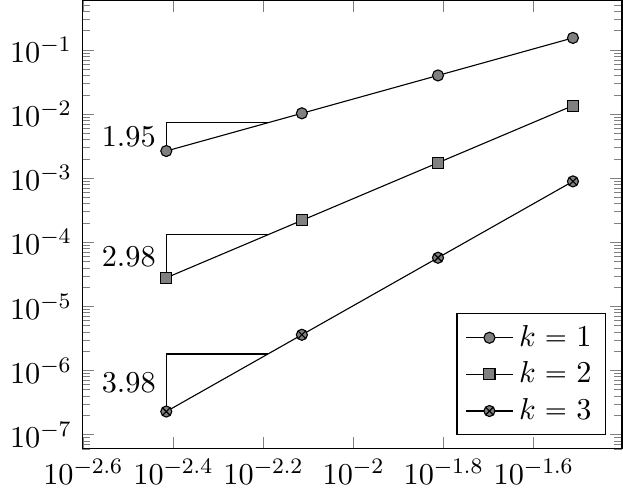}
    \subcaption{$\norm[a,h]{\uvu[h]^N-\Ih\vu^N}$, triangular}
  \end{minipage}
  \vspace{2mm} \\
  %% hexagonal mesh
  %% L2 pressure error 
  \begin{minipage}[b]{0.45\textwidth}
    \includegraphics[scale=0.9]{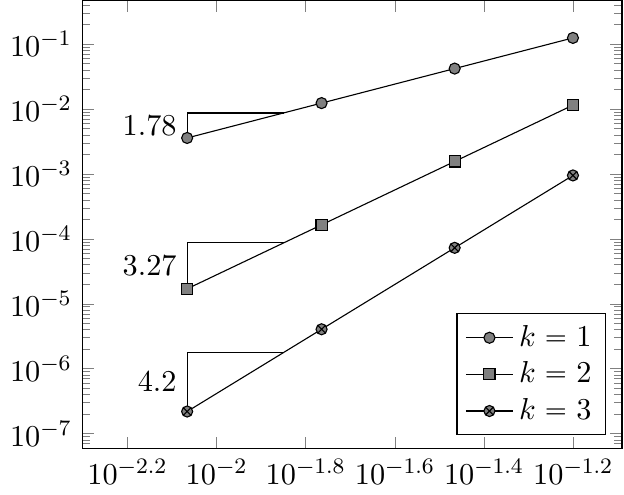}
    \subcaption{$\norm{p_h^N - \lproj{k} p^N}$, hexagonal}
  \end{minipage}
  %% Displacement energy error
  \begin{minipage}[b]{0.45\textwidth}
    \includegraphics[scale=0.9]{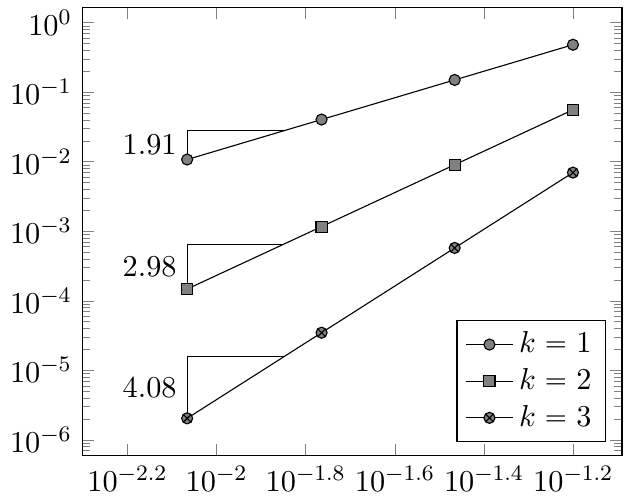}
    \subcaption{$\norm[a,h]{\uvu[h]^N-\Ih\vu^N}$, hexagonal}
  \end{minipage}
  \vspace{2mm} \\
  %% Voronoi mesh
  %% L2 pressure error 
  \begin{minipage}[b]{0.45\textwidth}
    \includegraphics[scale=0.4]{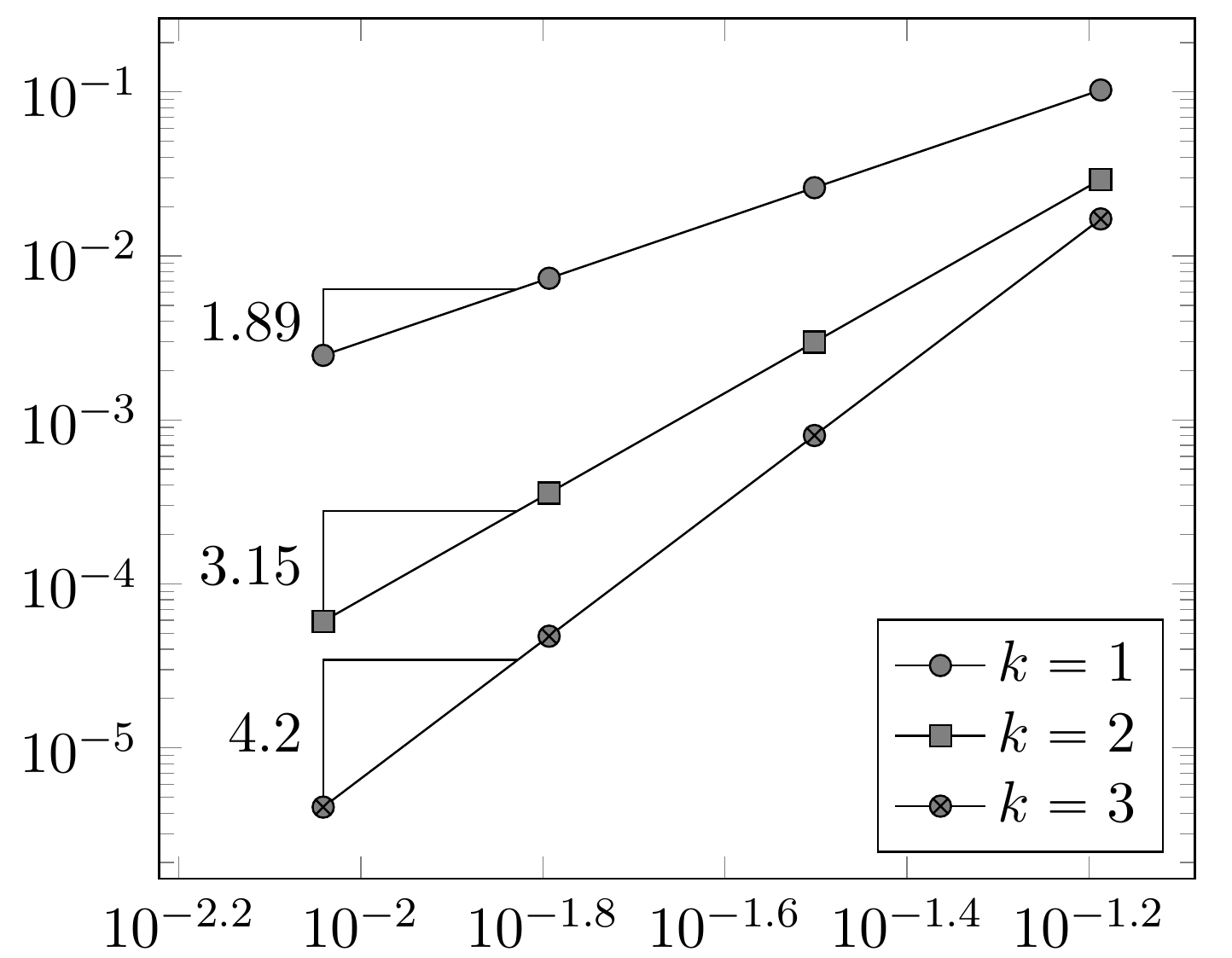}
    \subcaption{$\norm{p_h^N - \lproj{k} p^N}$, Voronoi}
  \end{minipage}
  %% Displacement energy error
  \begin{minipage}[b]{0.45\textwidth}
    \includegraphics[scale=0.4]{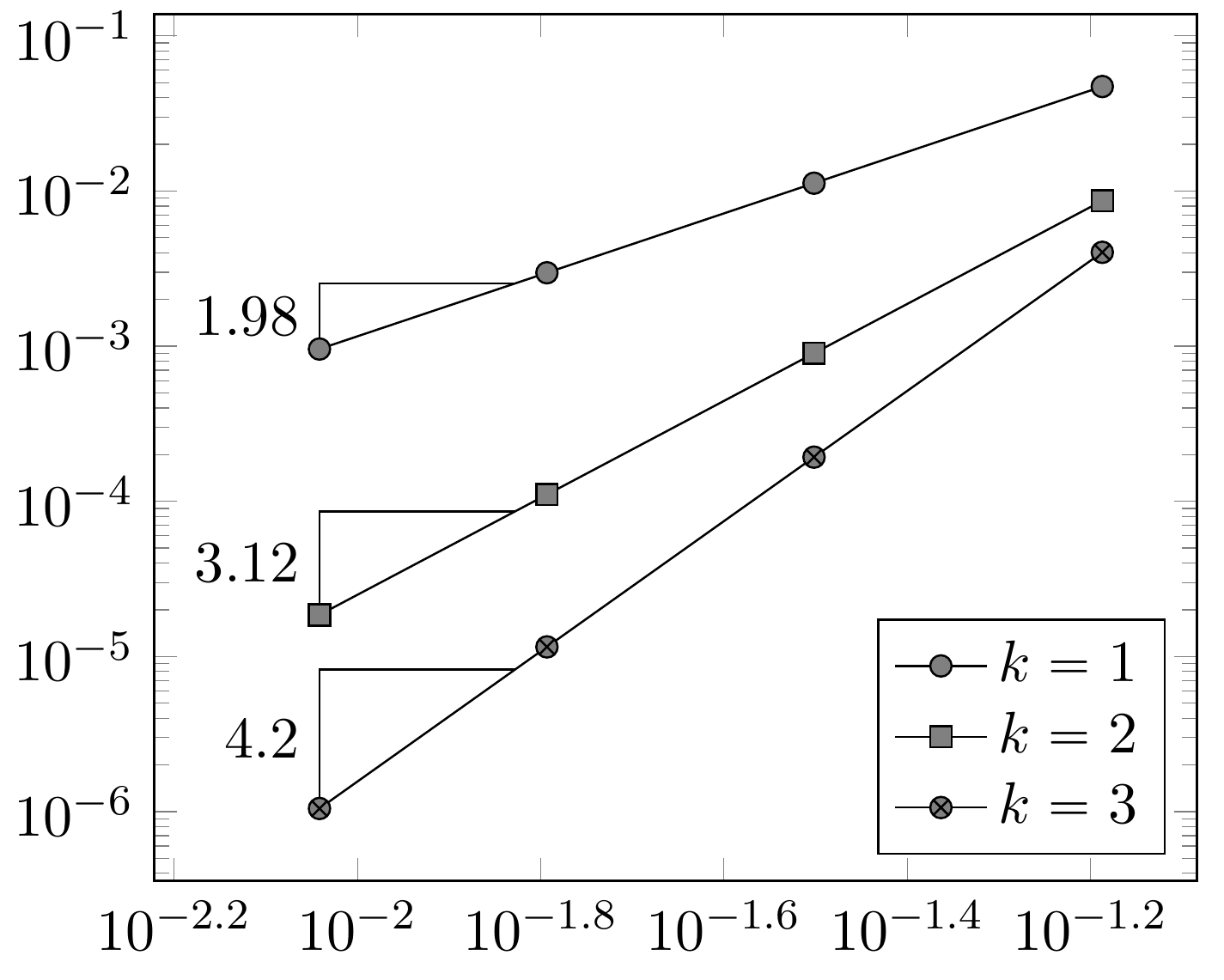}
    \subcaption{$\norm[a,h]{\uvu[h]^N-\Ih\vu^N}$, Voronoi}
  \end{minipage}
  \vspace{2mm} \\
  %% Nonmatching quadrangular
  %% L2 pressure error 
  \begin{minipage}[b]{0.45\textwidth}
    \includegraphics[scale=0.4]{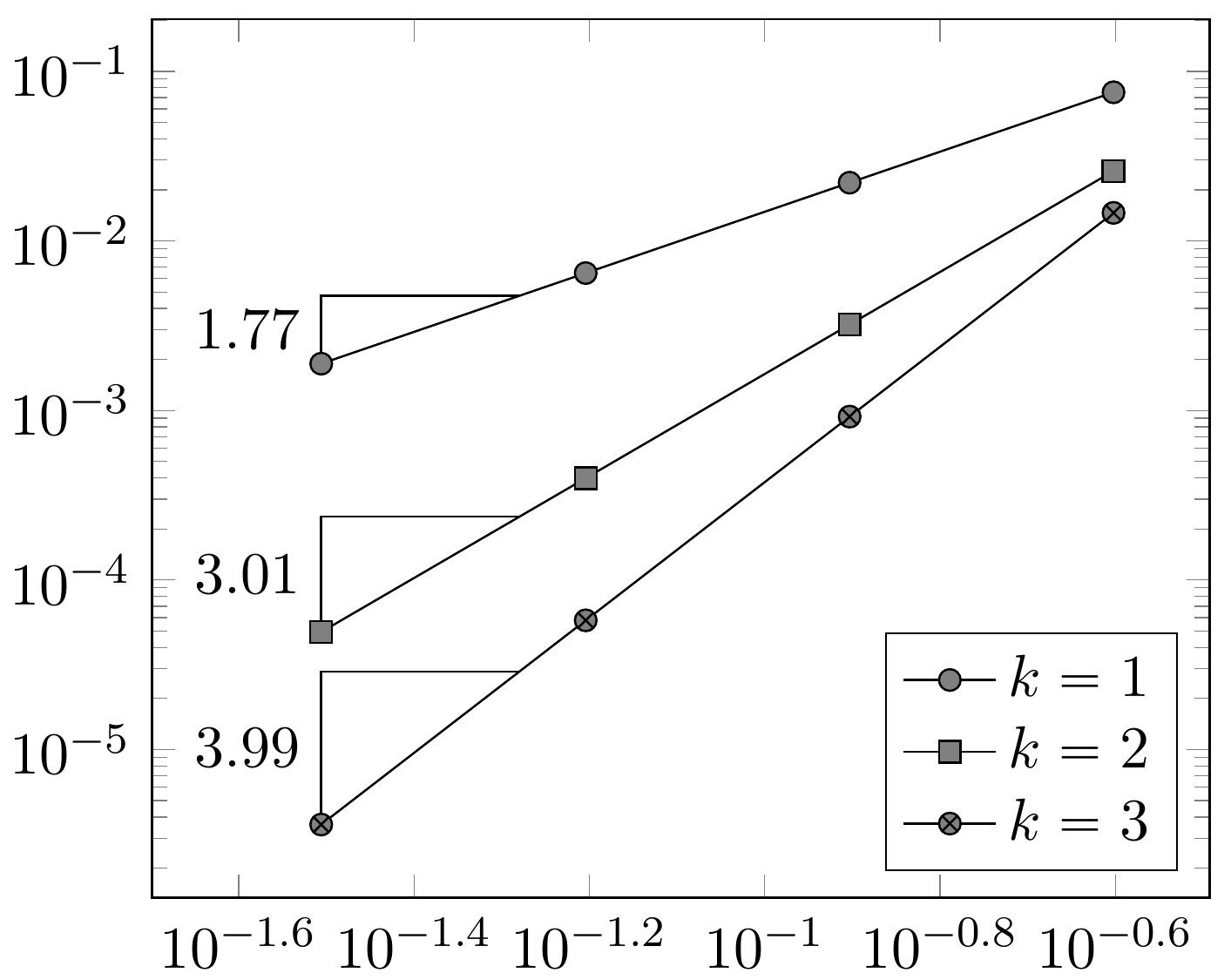}
    \subcaption{$\norm{p_h^N - \lproj{k} p^N}$, nonmatching}
  \end{minipage}
  %% Displacement energy error
  \begin{minipage}[b]{0.45\textwidth}
    \includegraphics[scale=0.4]{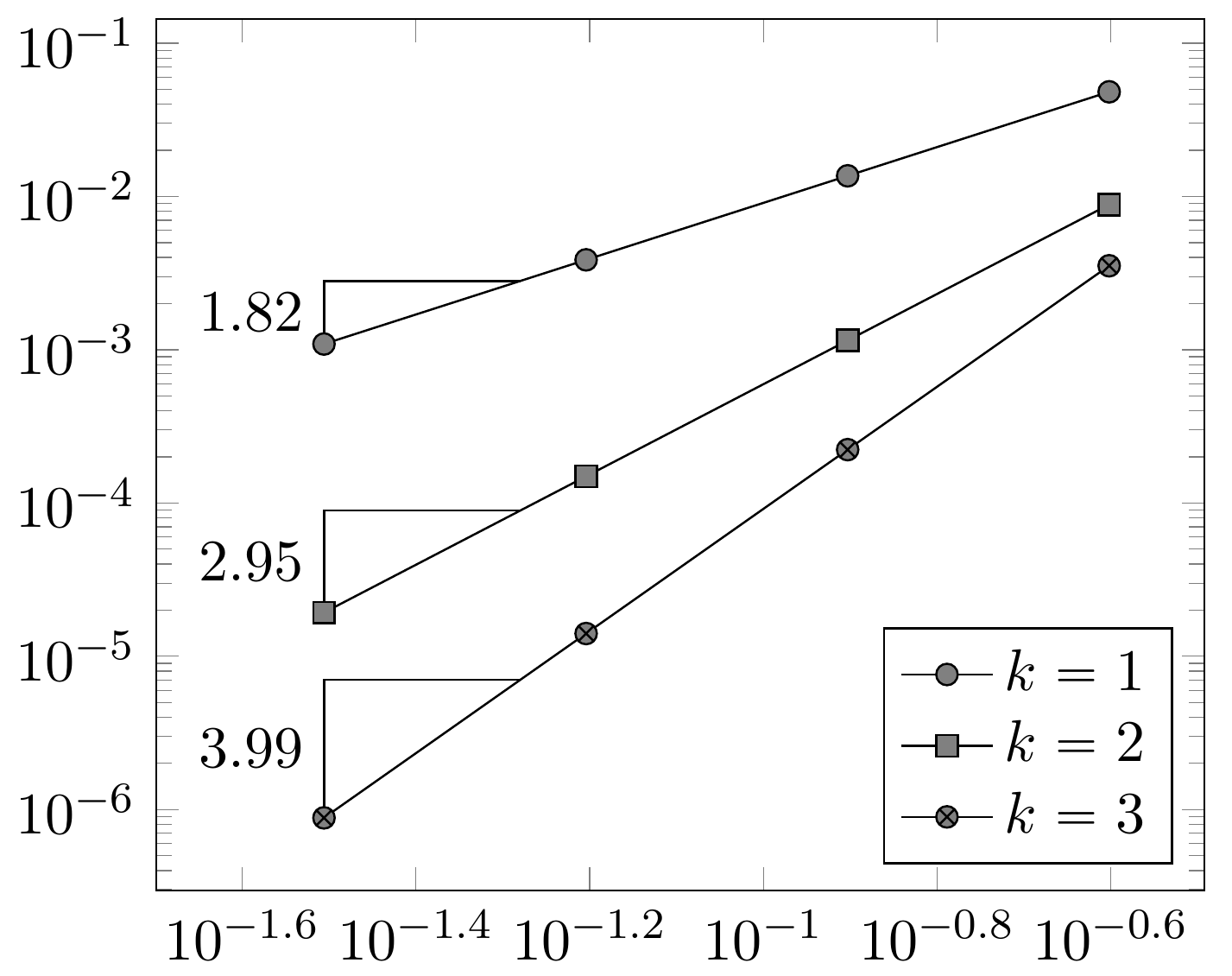}
    \subcaption{$\norm[a,h]{\uvu[h]^N-\Ih\vu^N}$, nonmatching}
  \end{minipage}
  \caption{Errors vs. $h$\label{fig:convergence}}
\end{figure}

The convergence in time was also separately checked considering the method with spatial degree $k=3$ on the hexagonal mesh with mesh size $h=0.0172$ and time step decreasing from $\tau = 0.1$ to $\tau = 0.0125$. With this choice, the time-component of the error is dominant, and Figure~\ref{fig:BDF2} confirms the second order convergence of the BDF2 scheme.

\begin{figure}
  \centering
  \includegraphics{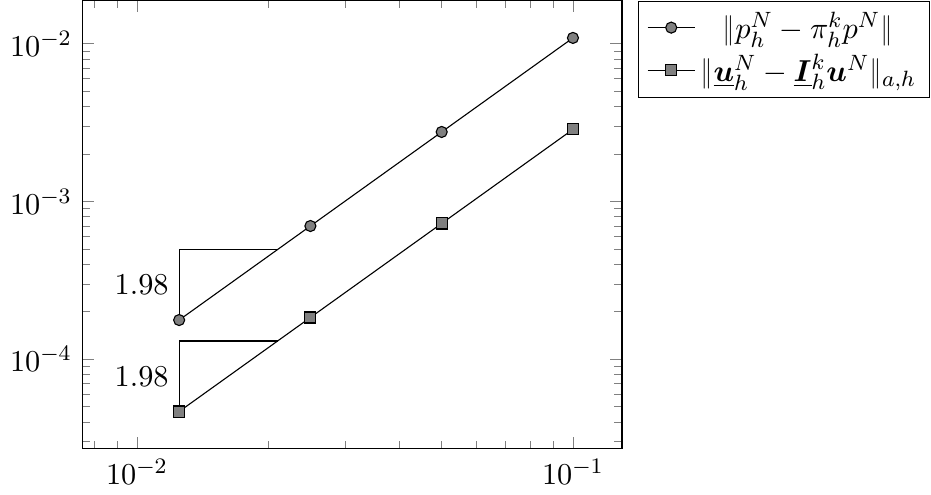}
  \caption{Time convergence rate with BDF2, hexagonal mesh\label{fig:BDF2}}
\end{figure}

\subsection{Barry and Mercer's test case}

A test case more representative of actual physical configurations is that of Barry and Mercer~\cite{Barry.Mercer:99}, for which an exact solution is available (we refer to the cited paper and also to~\cite[Section~4.2.1]{Phillips:05} for the its expression).
We let $\Omega=(0,1)^2$ and consider the following time-independent boundary conditions on $\partial\Omega$
$$
\vu\SCAL\vec{\tau}=0,\qquad
\normal\trans\GRAD\vu\normal=0,\qquad
p=0,
$$
where $\vec{\tau}$ denotes the tangent vector on $\partial\Omega$.
The evolution of the displacement and pressure fields is driven by a periodic pointwise source (mimicking a well) located at $\vec{x}_0=(0.25,0.25)$:
$$
g=\delta(\vec{x}-\vec{x}_0) \sin(\hat{t}),
$$
with normalized time $\hat{t}\eqbydef\beta t$ for $\beta\eqbydef(\lambda + 2\mu)\diff$.
As in~\cite{Phillips.Wheeler:07*1,Rodrigo.Gaspar.ea:15}, we use the following values for the physical parameters:
$$
c_{0} = 0,\qquad
E = \pgfmathprintnumber{1e+5},\qquad
\nu = 0.1,\qquad
\diff = \pgfmathprintnumber{1e-2},
$$
where $E$ and $\nu$ denote Young's modulus and Poisson ratio, respectively, and
$$
\lambda = \frac{E\nu}{(1+\nu)(1-2\nu)},\qquad
\mu = \frac{E}{2(1+\nu)}.
$$
In the injection phase $\hat{t}\in(0,\pi)$, we observe an inflation of the domain, which reaches its maximum at $\hat{t}=\nicefrac{\pi}{2}$; cf.~Figure~\ref{fig:pressure_deformed:pi/2}.
In the extraction phase $\hat{t}\in(\pi,2\pi)$, on the other hand, we have a contraction of the domain which reaches its maximum at $\hat{t}=\nicefrac{3\pi}{2}$; cf.~Figure~\ref{fig:pressure_deformed:3pi/2}.
\begin{figure}
  \centering
  \begin{minipage}{0.45\textwidth}
    \centering
    \includegraphics[height=4.5cm]{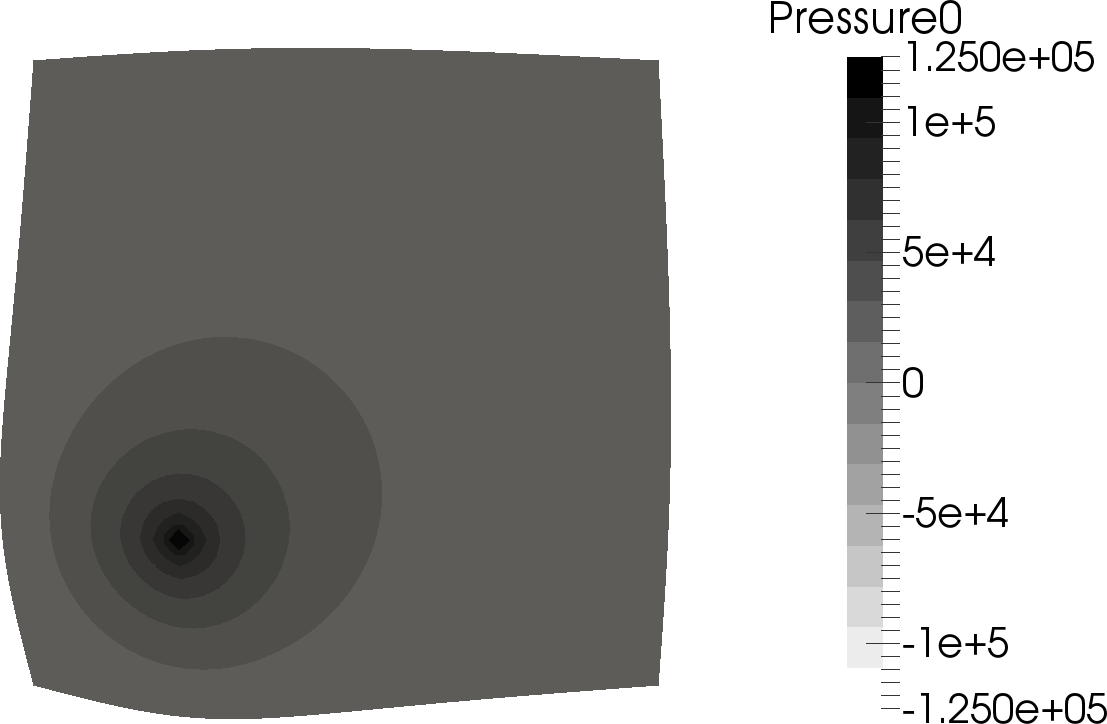}
    \subcaption{$\hat{t}=\nicefrac{\pi}{2}$\label{fig:pressure_deformed:pi/2}}
  \end{minipage}
  \hspace{0.5cm}
  \begin{minipage}{0.45\textwidth}
    \centering
    \includegraphics[height=4.5cm]{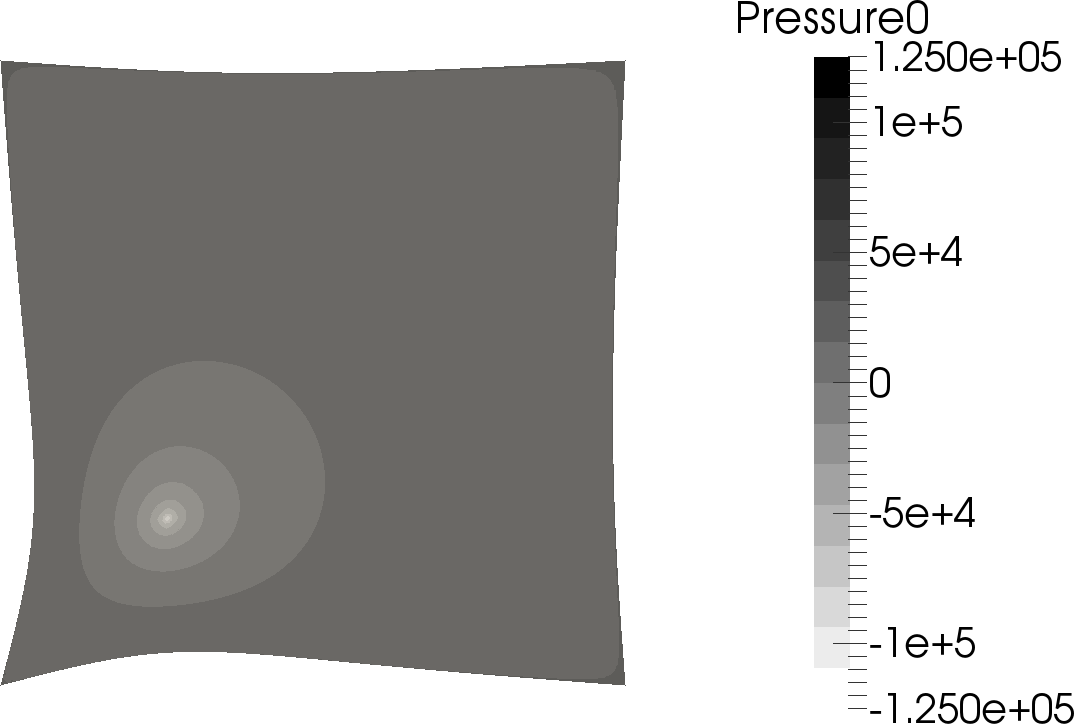}
    \subcaption{$\hat{t}=\nicefrac{3\pi}{2}$\label{fig:pressure_deformed:3pi/2}}
  \end{minipage}
  \caption{Pressure field on the deformed domain at different times for the finest Cartesian mesh containing $\pgfmathprintnumber{4192}$ elements\label{fig:pressure_deformed}}
\end{figure}

The following results have been obtained with the lowest-order version of the method corresponding to $k=1$ (taking advantage of higher orders would require local mesh refinement, which is out of the scope of the present work).
In Figure~\ref{fig:bm_press_prof} we plot the pressure profile at normalized times $\hat{t}=\nicefrac{\pi}{2}$ and $\hat{t}=\nicefrac{3\pi}{2}$ along the diagonal $(0,0)$--$(1,1)$ of the domain.
We consider two Cartesian meshes containing \pgfmathprintnumber{1024} and \pgfmathprintnumber{4096} elements, respectively, as well as two (predominantly) hexagonal meshes containing \pgfmathprintnumber{1072} and \pgfmathprintnumber{4192} elements, respectively.
In all the cases, a time step $\tau=\left(\nicefrac{2\pi}{\beta}\right)\cdot 10^{-2}$ is used.
We note that the behaviour of the pressure is well-captured even on the coarsest meshes.
For the finest hexagonal mesh, the relative error on the pressure in the $L^2$-norm at times $\hat{t}=\nicefrac{\pi}{2}$ and $\hat{t}=\nicefrac{3\pi}{2}$ is $2.85\%$.
\begin{figure}
  \centering
  \begin{minipage}{0.45\textwidth}
    \centering
    \includegraphics[height=4.5cm]{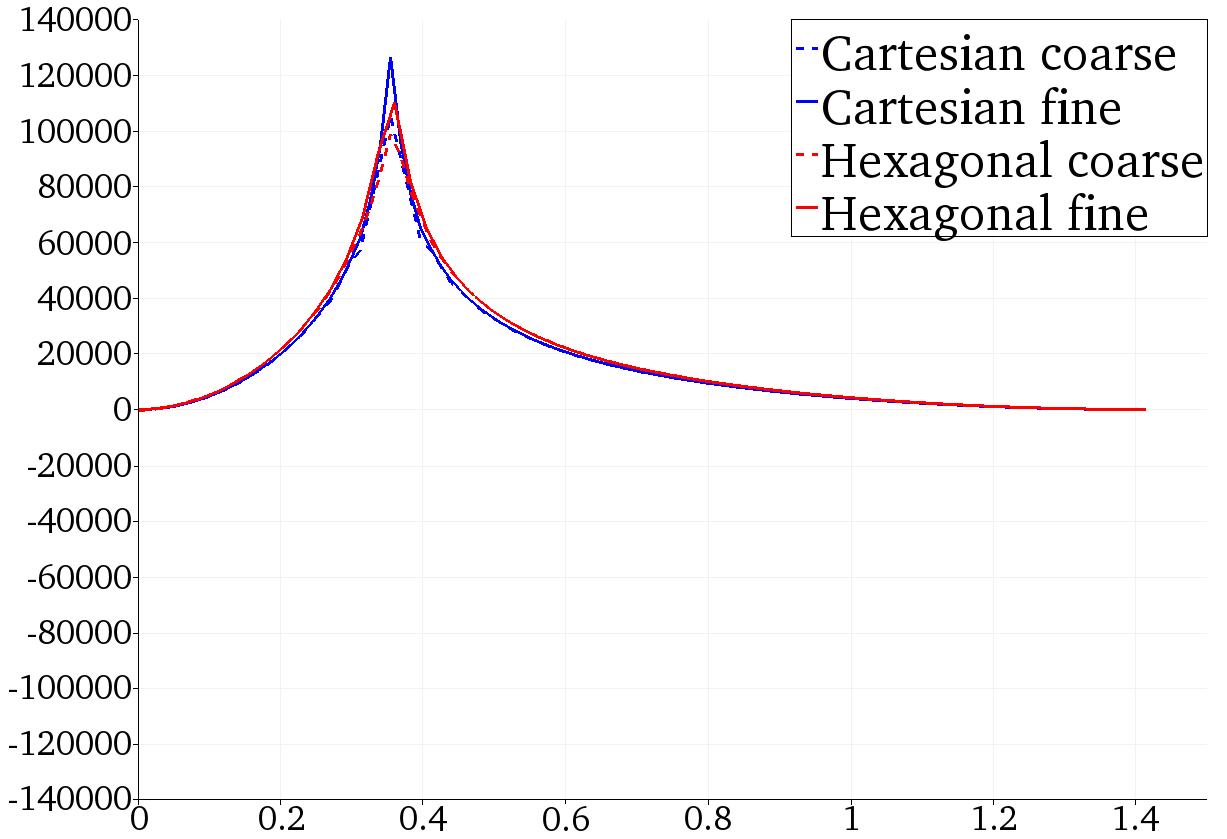}
    \subcaption{$\hat{t} = \nicefrac{\pi}{2}$, $\diff=\pgfmathprintnumber{1e-2}$\label{fig:bm_press_prof:pi/2}}
  \end{minipage}
  \hspace{0.5cm}
  \begin{minipage}{0.45\textwidth}
    \centering
    \includegraphics[height=4.5cm]{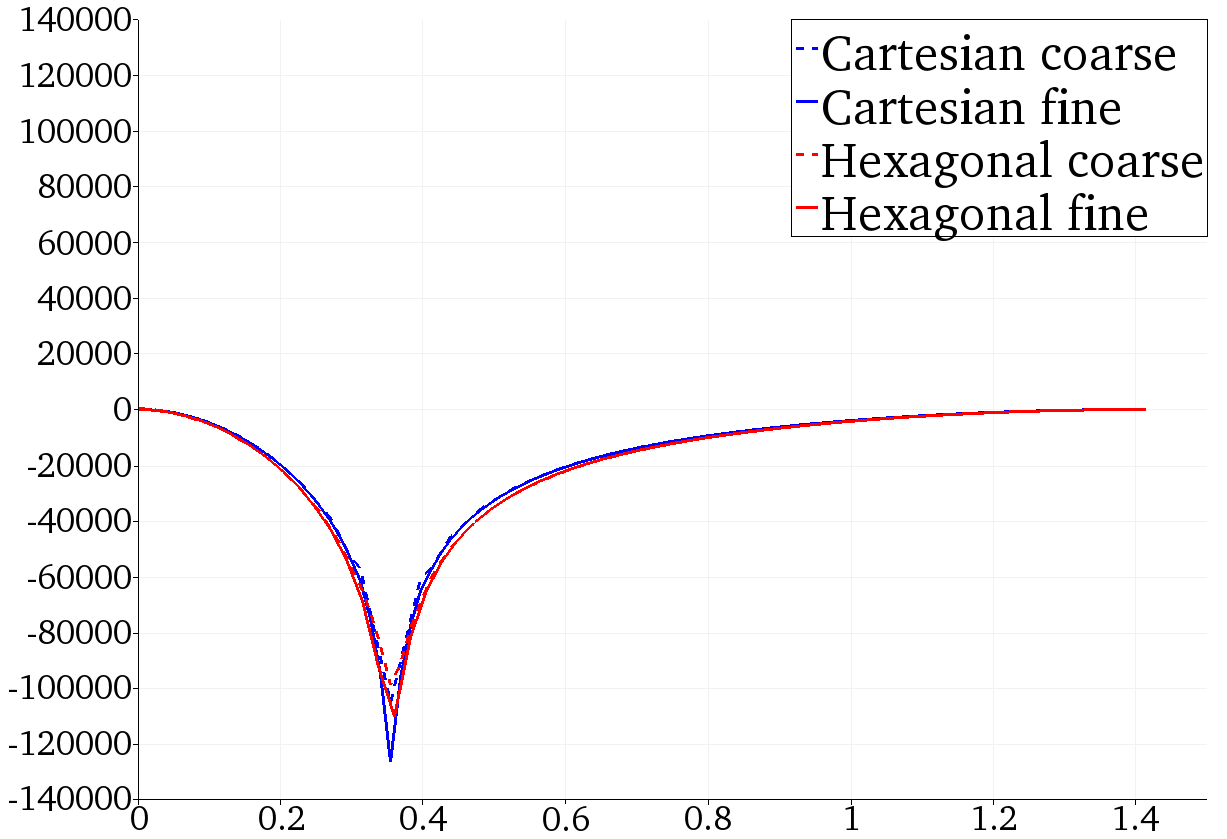}
    \subcaption{$\hat{t} = \nicefrac{3\pi}{2}$, $\diff=\pgfmathprintnumber{1e-2}$\label{fig:bm_press_prof:3pi/2}}
  \end{minipage}
  \caption{Pressure profiles along the diagonal $(0,0)$--$(1,1)$ of the domain for different normalized times $\hat{t}$ and meshes ($k=1$).
    The time step is here $\tau=\left(\nicefrac{2\pi}{\beta}\right)\cdot 10^{-2}$.\label{fig:bm_press_prof}}
\end{figure}

To check the robustness of the method with respect to pressure oscillations for small permeabilities combined with small time steps, we also show in Figure~\ref{fig:bm_press_prof:first_step} the pressure profile after one and two step with $\diff=\pgfmathprintnumber{1e-6}$ and $\tau=\pgfmathprintnumber{1e-4}$ on the Cartesian and hexagonal meshes with 
$\pgfmathprintnumber{4096}$ and $\pgfmathprintnumber{4192}$ elements, respectively. We remark that the first time step is performed using the backward Euler scheme, while the second with the second order BDF2 scheme.
This situation corresponds to the one considered in~\cite[Figure~5.10]{Rodrigo.Gaspar.ea:15} to highlight the onset of spurious oscillations in the pressure.
In our case, small oscillations can be observed for the Cartesian mesh (cf.~Figure~\ref{fig:bm_press_prof:first_step:cart} and~Figure~\ref{fig:bm_press_prof:second_step:cart}), whereas no sign of oscillations in present for the hexagonal mesh (cf.~Figure~\ref{fig:bm_press_prof:first_step:hexa} and~Figure~\ref{fig:bm_press_prof:second_step:hexa}).
One possible conjecture is that increasing the number of element faces contributes to the monotonicity of the scheme.
\begin{figure}
  \centering
  \begin{minipage}{0.45\textwidth}
    \includegraphics[height=4.6cm]{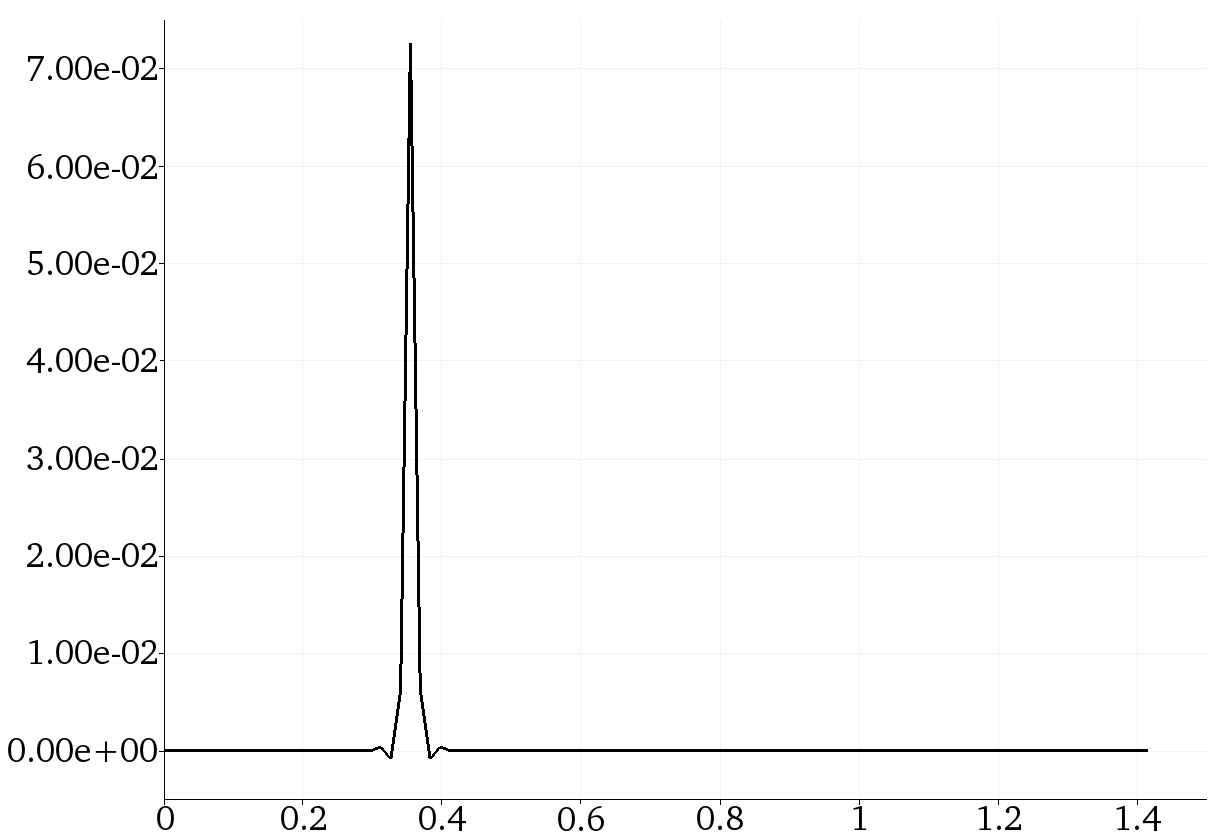}
    \subcaption{Cartesian mesh ($\card{\Th}=\pgfmathprintnumber{4028}$), first step\label{fig:bm_press_prof:first_step:cart}}
  \end{minipage}
  \hspace{0.5cm}
  \begin{minipage}{0.45\textwidth}
    \includegraphics[height=4.6cm]{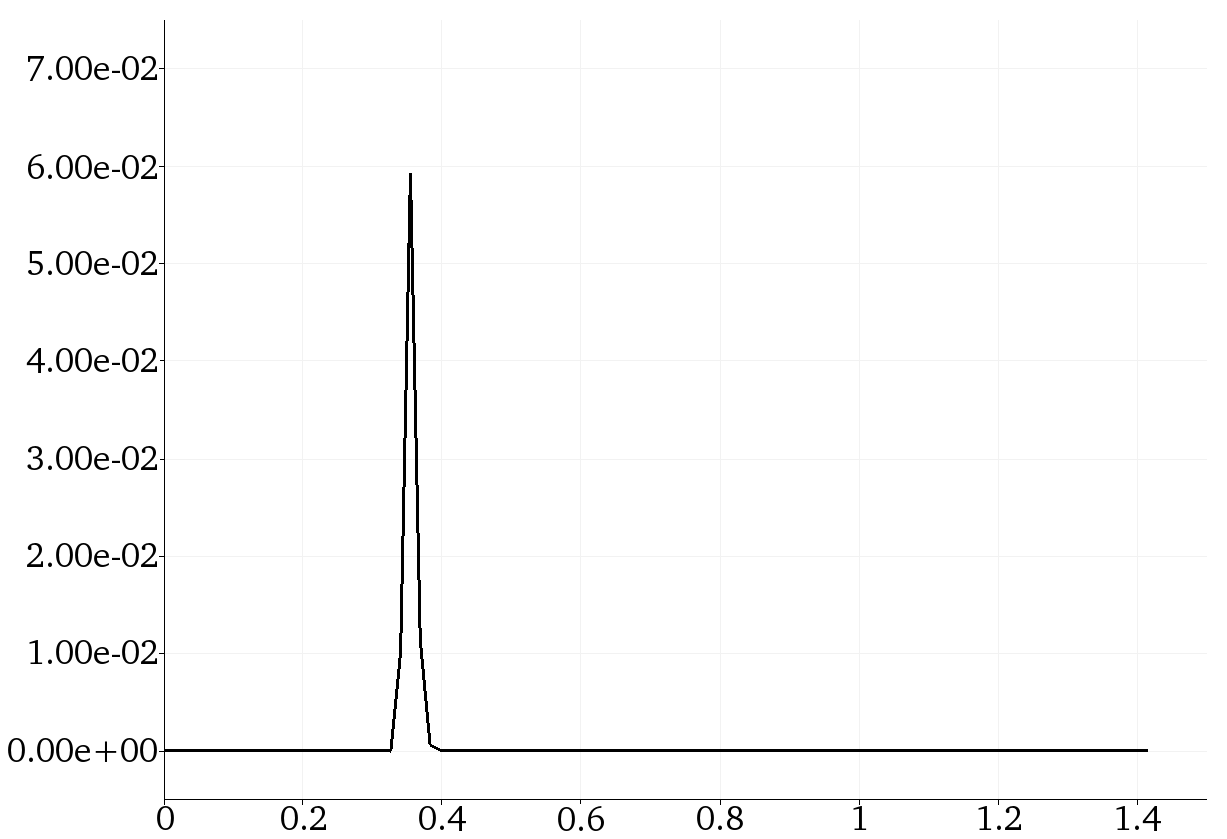}
    \subcaption{Hexagonal mesh ($\card{\Th}=\pgfmathprintnumber{4192}$), first step\label{fig:bm_press_prof:first_step:hexa}}
  \end{minipage} 
  \begin{minipage}{0.45\textwidth}
    \vspace{0.5cm}
    \includegraphics[height=4.6cm]{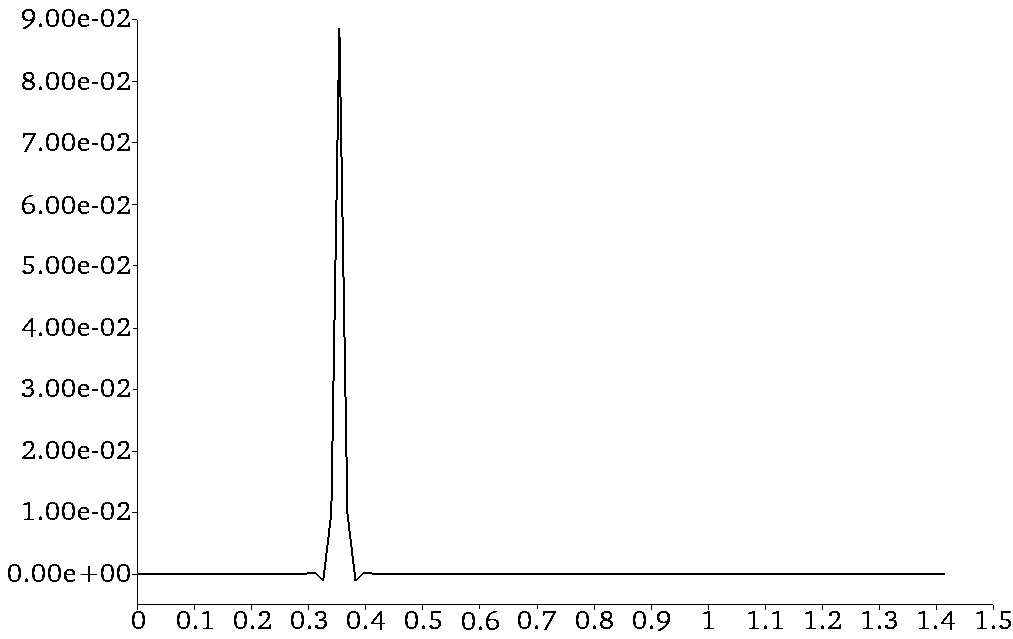}
    \subcaption{Cartesian mesh ($\card{\Th}=\pgfmathprintnumber{4028}$), second step\label{fig:bm_press_prof:second_step:cart}}
  \end{minipage}
  \hspace{0.5cm}
  \begin{minipage}{0.45\textwidth}
    \vspace{0.5cm}
    \includegraphics[height=4.6cm]{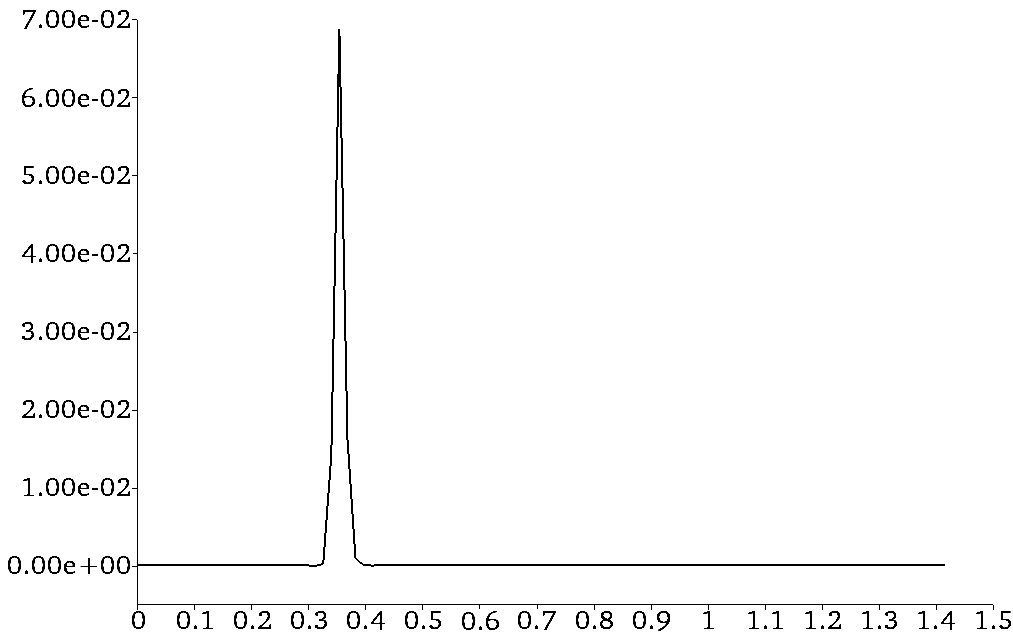}
    \subcaption{Hexagonal mesh ($\card{\Th}=\pgfmathprintnumber{4192}$), second step \label{fig:bm_press_prof:second_step:hexa}}
  \end{minipage} 
  \caption{Pressure profiles along the diagonal $(0,0)$--$(1,1)$ of the domain for $\diff=\pgfmathprintnumber{1e-6}$ and time step $\tau=\pgfmathprintnumber{1e-4}$. Small oscillations are present on the Cartesian mesh (left), whereas no sign of oscillations is present on the hexagonal mesh (right).\label{fig:bm_press_prof:first_step}}
\end{figure}

%------------------------------------------------------------------------------%

\paragraph{Acknowledgements}
The work of M. Botti was partially supported by Labex NUMEV (ANR-10-LABX-20) ref. 2014-2-006.
The work of D. A. Di Pietro was partially supported by project HHOMM (ANR-15-CE40-0005).

%------------------------------------------------------------------------------%

\appendix

\section{Flux formulation}\label{sec:flux.form}

In this section, we reformulate the discrete problem~\eqref{eq:biot.h} to unveil the local conservation properties of the method.
Before doing so, we need to introduce a few operators and some notation to treat the boundary terms.

We start from the mechanical equilibrium.
Let an element $T\in\Th$ be fixed and denote by $\UpT\eqbydef\Poly[d-1]{k}(\Fh[T])^{d}$ the broken polynomial space of degree $\le k$ on the boundary $\partial T$ of $T$.
We define the boundary operator $\LT:\UpT\to\UpT$ such that, for all $\vphi\in\UpT$,
\begin{equation}\label{eq:LT}
  \restrto{\LT\vphi}{F}
  \eqbydef\lproj[F]{k}\left(
    \restrto{\vphi}{F} - \rT(\vec{0}, (\restrto{\vphi}{F})_{F\in\Fh[T]}) + \lproj[T]{k}\rT(\vec{0}, (\restrto{\vphi}{F})_{F\in\Fh[T]})
  \right)\qquad\forall F\in\Fh[T].
\end{equation}
We also need the adjoint $\LT[k,*]$ of $\LT$ such that
\begin{equation}\label{eq:LTs}
  \forall\vphi\in\UpT,\qquad
  (\LT\vphi,\vpsi)_{\partial T} = (\vphi,\LT[k,*]\vpsi)_{\partial T}
  \qquad\forall\vpsi\in\UpT.
\end{equation}
For a collection of DOFs $\uvv[T]\in\UT$, we denote in what follows by $\vv[\partial T]\in\UpT$ the function in $\UpT$ such that $\restrto{{\vv[\partial T]}}{F}=\vv[F]$ for all $F\in\Fh[T]$.
Finally, it is convenient to define the discrete stress operator $\ST:\UT\to\Poly{k}(T)^{d\times d}$ such that, for all $\uvv[T]\in\UT$,
\begin{equation}\label{eq:ST}
  \ST\uvv[T]\eqbydef 2\mu\GRADs\rT\uvv[T] + \lambda\Id\DT\uvv[T].
\end{equation}

To reformulate the mass conservation equation, we need to introduce the classical lifting operator $\Rh:\Ph\to\Poly{k-1}(\Th)^{d}$ such that, for all $q_{h}\in\Ph$, it holds
\begin{equation}\label{eq:Rh}
  (\Rh q_{h},\vec{\xi}_{h})
  = \sum_{F\in\Fhi} (\jump{q_{h}}, \wavg{\diff\vec{\xi}_{h}}\SCAL\normal_{F})_{F}\qquad
  \forall\vec{\xi}_{h}\in\Poly{k-1}(\Th)^{d}.
\end{equation}

\begin{lemma}[Flux formulation of problem~\eqref{eq:biot.h}]\label{lem:flux_form}
  Problem~\eqref{eq:biot.h} can be reformulated as follows: Find $(\uvu^{n},p_{h}^{n})\in\UhD\times\Ph$ such that it holds, for all $(\uvv,q_{h})\in\UhD\times\Poly{k}(\Th)$ and all $T\in\Th$,
  \begin{subequations}\label{eq:biot.flux}
    \begin{align}
      \label{eq:biot.flux:mech}
      (\ST\uvu[T]^{n} - p_{h}^{n}\Id,\GRADs\vv[T])_{T}
      + \!\! \sum_{F\in\Fh[T]}(\Flux(\uvu[T]^{n},\restrto{p_{h}^{n}}{T}), \vv[F]-\vv[T])_{F} &= (\vf^{n},\vv[T])_{T},
      \\
      \label{eq:biot.flux:flow}
      (c_0\ddt p_h^n,q_h)_T
      - (\ddt\vu[T]^{n} - \diff(\GRADh p_{h}^{n} - \Rh p_{h}^{n}),\GRADh q_{h} )_{T}
      - \!\! \sum_{F\in\Fh[T]}(\flux(\ddt\vu[F]^{n},p_{h}^{n}),\restrto{q_{h}}{T})_{F}
      &= (g^{n},q_{h}),
    \end{align}
  \end{subequations}
  where, for all $T\in\Th$ and all $F\in\Fh[T]$, the numerical traction $\Flux:\UT\times\Poly{k}(T)\to\Poly[d-1]{k}(F)^{d}$ and mass flux $\flux:\Poly[d-1]{k}(F)^d\times\Poly{k}(\Th)\to\Poly[d-1]{k}(F)$ are such that
  \begin{equation}\label{eq:fluxes}
    \begin{aligned}
      \Flux(\uvv[T],q) &\eqbydef 
      \big( \ST\uvv[T] - q\Id \big) \normal_{TF}
      + (2\mu)\LT[k,*](\fh^{-1}\LT(\vv[\partial T] - \vv[T])),
      \\
      \flux(\vv[F],q_{h}) &\eqbydef
      \begin{cases}
        \big( -\vv[F]^{n} + \wavg{\diff\GRADh q_{h}} \big)\SCAL\normal_{TF}
        - \frac{\varsigma\lambda_{\diff,F}}{h_{F}}\jump{q_{h}}(\normal_{TF}\SCAL\normal_F)
        &\text{if $F\in\Fhi$},
        \\
        0 &\text{otherwise,}
      \end{cases}
    \end{aligned}
  \end{equation}
  with $\fh\in\Poly{0}(\Fh[T])$ such that $\restrto{\fh}{F}=h_{F}$ for all $F\in\Fh[T]$, and it holds, for all $F\in\Fhi$ such that $F\in\Fh[T_{1}]\cap\Fh[T_{2}]$,
  \begin{subequations}\label{eq:flux.cons}
    \begin{align}
      \label{eq:flux.cons:mech}
      \Flux[T_{1}F](\uvu[T_{1}]^{n},\restrto{p_{h}^{n}}{T_{1}})
      + \Flux[T_{2}F](\uvu[T_{2}]^{n},\restrto{p_{h}^{n}}{T_{2}}) &= \vec{0}
      \\
      \label{eq:flux.cons:flow}
      \flux[T_{1} F](\ddt\vu[F]^{n},p_{h}^{n}) 
      + \flux[T_{2}F](\ddt\vu[F]^{n},p_{h}^{n}) &= 0.
    \end{align}
  \end{subequations}
\end{lemma}

\begin{proof}
(1) \emph{Proof of~\eqref{eq:biot.flux:mech}.}
Proceeding as in~\cite[Section~3.1]{Cockburn.Di-Pietro.ea:15}, the stabilization bilinear form $s_{T}$ defined by~\eqref{eq:sT} can be rewritten as
$$
  s_{T}(\uvw[T],\uvv[T])
  = \sum_{F\in\Fh[T]} (\LT[k,*](\fh^{-1}\LT(\vw[\partial T]-\vw[T])),\vv[F]-\vv[T])_{F}.
$$
Therefore, using the definitions~\eqref{eq:rT} of $\rT\uvv[T]$ with $\vw=\rT\uvu[T]^{n}$ and~\eqref{eq:DT} of $\DT\uvv[T]$ with $q=\restrto{p_{h}^{n}}{T}$, and recalling the definition~\eqref{eq:ST} of $\ST$, one has
\begin{equation}\label{eq:flux:mech:aT}
    a_{T}(\uvu[T]^{n},\uvv[T])
    = (\ST\uvu[T]^{n},\GRADs\vv[T])_{T}
    + \!\! \sum_{F\in\Fh[T]} (\ST\uvu[T]^{n}\normal_{TF} + (2\mu) \LT[k,*](\fh^{-1}\LT(\vu[\partial T]^{n} - \vu[T]^{n})), \vv[F] - \vv[T])_{F}. 
\end{equation}
On the other hand, using again the definition~\eqref{eq:DT} of $\DT\uvv[T]$ with $q=\restrto{p_{h}^{n}}{T}$, one has
\begin{equation}\label{eq:flux:mech:bT}
    b_{T}(\uvv[T],\restrto{p_{h}^{n}}{T})
    = -(p_{h}^{n}\Id,\GRADs\vv[T])_{T} - \sum_{F\in\Fh[T]}
(\restrto{p_{h}^{n}}{T}\normal_{TF}, \vv[F]-\vv[T])_{F}.
\end{equation}
Equation~\eqref{eq:biot.flux:mech} follows summing~\eqref{eq:flux:mech:aT} and~\eqref{eq:flux:mech:bT}.
\\
(2) \emph{Proof of~\eqref{eq:biot.flux:flow}.}
Using the definition~\eqref{eq:DT.bis} of $\DT$ with $\uvv[T]=\ddt\uvu[T]^{n}$ and $q=\restrto{q_{h}}{T}$, it is inferred that
\begin{equation}\label{eq:flux:flow:bT}
  b_{T}(\ddt\uvu[T]^{n},q_{h}) = 
  -(\ddt\vu[T]^{n}, \GRADh q_{h})_{T}
  + \sum_{F\in\Fh[T]}  (\ddt\vu[F]^{n}\SCAL\normal_{TF}, \restrto{q_{h}}{T})_{F}.
\end{equation}
On the other hand, adapting the results~\cite[Section~4.5.5]{Di-Pietro.Ern:12} to the homogeneous Neumann boundary condition~\eqref{eq:biot.strong:bc.p}, it is inferred
\begin{multline}\label{eq:flux:flow:cT}
  c_{h}(p_{h}^{n},q_{h})
  = \sum_{T\in\Th}\Bigg\{
    (\diff(\GRADh p_{h}^{n} - \Rh p_{h}^{n})\SCAL\GRADh q_{h})_{T}
    \\
    - \sum_{F\in\Fh[T]\cap\Fhi} (\wavg{\diff\GRADh p_{h}^{n}}\SCAL\normal_{TF} - \frac{\varsigma\lambda_{\diff,F}}{h_{F}}\jump{p_{h}^{n}}(\normal_{TF}\SCAL\normal_F), \restrto{q_{h}}{T})_{F}
  \Bigg\}.
\end{multline}
Equation~\eqref{eq:biot.flux:flow} follows summing~\eqref{eq:flux:flow:bT} and~\eqref{eq:flux:flow:cT}.
\\
(3) \emph{Proof of~\eqref{eq:flux.cons}.} To prove~\eqref{eq:flux.cons:mech}, let an internal face $F\in\Fhi$ be fixed, and make $\uvv[h]$ in~\eqref{eq:flux.cons:mech} such that $\vv[T]\equiv\vec{0}$ for all $T\in\Th$, $\vv[F']\equiv\vec{0}$ for all $F'\in\Fh\setminus\{F\}$, let $\vv[F]$ span $\Poly[d-1]{k}(F)$ and rearrange the sums.
The mass flux conservation~\eqref{eq:flux.cons:flow} follows immediately from the expression of $\flux$ observing that, for all $(\uvv[h],q_h)\in\Uh\times\Ph$ and all $F\in\Fhi$, the quantity 
$$
\big(-\vv[F] + \wavg{\diff\GRADh q_h} \big)\SCAL\normal_F - \frac{\varsigma\lambda_{\diff,F}}{h_F}\jump{q_h}
$$
is single-valued on $F$.
\end{proof}

Let now an element $T\in\Th$ be fixed.
Choosing as test functions in~\eqref{eq:biot.flux:mech} $\uvv\in\UhD$ such that $\vv[F]\equiv\vec{0}$ for all $F\in\Fh$, $\vv[T']\equiv\vec{0}$ for all $T'\in\Th\setminus\{T\}$, and $\vv[T]$ spans $\Poly{k}(T)^{d}$, we infer the following local mechanical equilibrium relation:
For all $\vv[T]\in\Poly{k}(T)^{d}$,
$$
(\ST\uvu[T]^{n} - p_{h}^{n}\Id,\GRADs\vv[T])_{T}
- \sum_{F\in\Fh[T]}(\Flux(\uvu[T]^{n},\restrto{p_{h}^{n}}{T}), \vv[T])_{F} = (\vf^{n},\vv[T])_{T}.
$$
Similarly, selecting $q_{h}$ in~\eqref{eq:biot.flux:flow} such that $\restrto{q_{h}}{T'}\equiv 0$ for all $T'\in\Th\setminus\{T\}$ and $q_{T}\eqbydef\restrto{q_{h}}{T}$ spans $\Poly{k}(T)$, we infer the following local mass conservation relation:
For all $q_{T}\in\Poly{k}(T)$,
$$
(c_0\ddt p_h^n,q_T)_T
- (\ddt\vu[T]^{n} - \diff(\GRADh p_{h}^{n} - \Rh p_{h}^{n}),\GRAD q_T )_{T}
- \sum_{F\in\Fh[T]}(\flux(\ddt\vu[F]^{n},p_{h}^{n}),q_T)_{F}
= (g^{n},q_T).
$$
To actually compute the numerical fluxes defined by~\eqref{eq:fluxes}, besides the operator $\ST$ defined by~\eqref{eq:ST} (which is readily available once $\rT$ and $\DT$ have been computed; cf.~\eqref{eq:rT} and~\eqref{eq:DT}, respectively), one also needs to compute the operators $\LT$ and $\LT[k,*]$. The latter operation can be performed at marginal cost, since it only requires to invert the face mass matrices of $\Poly[d-1]{k}(F)$ for all $F\in\Fh[T]$.

%------------------------------------------------------------------------------%

\begin{footnotesize}
  \bibliographystyle{plain}
  \bibliography{bho}
\end{footnotesize}

\end{document}